\documentclass{article}

\usepackage{amsmath}
\usepackage{amssymb}
\usepackage{booktabs}

\usepackage[amsmath,thref,hyperref,thmmarks]{ntheorem} % Max: Using ntheorem. 'thmmarks' handles automatic QED symbols, 'amsmath' fixes symbol placement in equations, and 'thref' enables \thref commands.
\usepackage{amscd}
\usepackage[all]{xy} % Max: Changed to 'xy' because 'xypic' is a deprecated alias that triggers warnings. '[all]' loads all diagram features.
\usepackage{url}
\usepackage[shortlabels]{enumitem} % Max: Added this package. Added 'shortlabels' for backward compatibility.
\usepackage{tensor}
\usepackage[usenames,dvipsnames]{color}

\usepackage{graphicx}     % Max: Added this, it's needed for \HobbyCurve
\usepackage{cellspace}    % Space under hline in table
\usepackage{adjustbox}    % Max: Added this, it's needed for \HobbyCurve
\usepackage{microtype}      % Max: This package makes the final PDF look much more professional (better spacing, etc.)
\usepackage{hyperref} % For clickable refs/citations. Needs to be loaded late.
\usepackage{fvextra}

\usepackage{placeins}

\usepackage{tikz}
\usetikzlibrary{calc}
\usetikzlibrary{positioning}
\usetikzlibrary{hobby}

\theoremstyle{plain}
\theorembodyfont{\slshape}
\newtheorem{theorem}{Theorem}[section]
\newtheorem{corollary}[theorem]{Corollary}
\newtheorem{lemma}[theorem]{Lemma}
\newtheorem{proposition}[theorem]{Proposition}

\theorembodyfont{\upshape}
\theoremheaderfont{\bfseries}
\newtheorem{definition}[theorem]{Definition}
\newtheorem{problem}{Problem}
\newtheorem{remark}[theorem]{Remark}

\theoremstyle{nonumberplain}      % Proofs are not numbered
\theoremheaderfont{\itshape}      % Header "Proof" is Italic
\theorembodyfont{\normalfont}     % Body text is normal (not slanted)
\theoremseparator{.}              % Dot after header
\theoremsymbol{\ensuremath{\Box}} % Box symbol at the end

\theorempostwork{%
  \setcounter{case}{0}%
  \setcounter{step}{0}%
}

\newtheorem{proof}{Proof}

\newcommand{\qedhere}{\hfill\proofSymbol\NoEndMark}
\theorempostwork{} 
\theoremsymbol{}

\theoremstyle{plain}
\theoremheaderfont{\scshape} 
\theorembodyfont{\upshape}

\theoremstyle{plain}
\theoremheaderfont{\itshape}
\theorembodyfont{\upshape}

\makeatletter
\DeclareRobustCommand\em{%
  \@nomath\em                      % Error if used in math mode
  \ifdim \fontdimen\@ne\font >\z@  % Check if current font is slanted/italic
    \upshape                       % If yes: switch to UPRIGHT
  \else
    \slshape                       % If no: switch to SLANTED
  \fi
}
\makeatother

\setlist[enumerate,1]{label=(\roman*)} % Max: This is the new 'enumitem' way to set the default label.
\setlist[description]{font=\normalfont\emph} % Max: Customizing 'description' lists. This makes the labels emphasized (\emph) instead of the default bold (\textbf).

\makeatletter
\renewcommand{\subsection}{\@startsection{subsection}{2}%
	{\z@}{-2.25ex plus -1ex minus-.2ex}{-1em}{\bf}}
\makeatother

\renewcommand{\arraystretch}{1.5}

\newcommand{\NN}{\mathbb{N}}

\newcommand{\QQ}{\mathbb{Q}}

\newcommand{\PP}{\mathbb{P}}

\DeclareMathOperator{\GL}{GL}
\DeclareMathOperator{\PGL}{PGL}
\DeclareMathOperator{\SL}{SL}
\DeclareMathOperator{\Gal}{Gal}
\DeclareMathOperator{\Spec}{Spec}

\DeclareMathOperator{\Proj}{Proj}

\DeclareMathOperator{\Aut}{Aut}

\DeclareMathOperator{\Id}{Id}
\DeclareMathOperator{\Char}{char} % Max: Safe version, overwriting \char causes errors.

\newcommand{\OO}{\mathcal{O}}

\newcommand{\X}{\mathcal{X}}

\newcommand{\M}{\mathcal{M}}

\newcommand{\Hc}{\mathcal{H}}

\newcommand{\Xb}{\bar{X}}

\newcommand{\phib}{\bar{\phi}}

\newcommand{\nr}{{\mathrm{nr}}}

\newcommand{\gen}[1]{\mathopen\langle#1\mathclose\rangle}
\newcommand{\abs}[1]{\lvert#1\rvert}

\newcommand{\lexp}[2]{\tensor[^#1]{#2}{}}

\input{curve_library.tex}

\newcommand{\LL}{\mathcal{L}}       % Line bundle L
\newcommand{\LLb}{\bar{\LL}}   % Reduction of line bundle L
\DeclareMathOperator{\Res}{Res}     % Residue map
\DeclareMathOperator{\Tr}{Tr}       % Trace map
\DeclareMathOperator{\Cl}{Cl}
\begin{document}

\title{Semistable reduction of smooth quartics%
}

\author{Max Schwegele \and Kletus Stern \and Stefan Wewers}

\date{\today}

\maketitle

\begin{abstract}
We develop a method for computing stable reduction of smooth plane quartics over discretely
valued fields, including residue characteristic \(p=2\).
The method uses the GIT-semistable plane models constructed in an earlier part
of this project, together with an intrinsic description of hyperelliptic stable
curves, to characterize when the stable model is obtained from a GIT-stable
plane model by resolving its cusps.	More precisely,
for a smooth non-hyperelliptic curve of genus \(3\) with semistable reduction,
we show that it admits a GIT-stable plane model if and only if its stable
reduction is non-hyperelliptic. In that case, the stable model is obtained from
the GIT-stable plane model by replacing each cusp by a \(1\)-tail. Together with the companion paper on explicit local stable resolution of cusps,
this gives an effective approach to computing stable reduction of smooth plane
quartics. The resulting algorithms are implemented in the SageMath package
\texttt{StabilityFunction}.
\end{abstract}

%\begin{abstract}
%We study semistable reduction of smooth plane quartic curves over discretely valued fields, including residue characteristic $p=2$.
%Building on~\cite{SternWewers}, we use GIT to compute plane models whose special fibers are close to stable, providing a practical route from an explicit equation to stable reduction in the non-hyperelliptic case.
%
%We prove that a smooth plane quartic admits a unique GIT-stable plane model if and only if its stable reduction is non-hyperelliptic.
%When this holds, the stable model is obtained from the GIT-stable plane model by a canonical local modification supported at the cusps of the special fiber.
%This description isolates the geometric mechanism behind the hyperelliptic/non-hyperelliptic dichotomy.
%The explicit cusp-resolution step is carried out in the companion paper~\cite{cusp_resolution} and is implemented in \cite{KletusGitHub}.
%\end{abstract}	

% !TeX root = reduction_of_plane_quartics.tex

\section*{Introduction}

Let $K$ denote a field which is complete with respect to a discrete valuation $v_K$. Let $\OO_K$ denote the ring of integers and $k$ the residue field of $v_K$. Let $X$ be a smooth and absolutely irreducible curve over $K$, of genus $g\geq 2$.

By the Semistable Reduction Theorem, there exists a finite extension $L/K$ and an integral model $\X$ of $X_L:=X\otimes_K L$ which is \emph{semistable} in the sense of \cite{DeligneMumford69}. We may assume that the extension $L/K$ is Galois, and that the tautological action of the Galois group $G:=\Gal(L/K)$ on $X_L$ extends to the semistable model $\X$, inducing an action on the special fiber $\X_s$. This is the case, for instance, if $\X$ is the {\em stable model} of $X_L$, which is the unique minimal semistable model.

\begin{problem} \label{prob:ss-red}
    For an explicitly given curve $X/K$, compute a suitable extension $L/K$, a semistable model $\X$ of $X_L$ and the action of $\Gal(L/K)$ on $\X_s$.
\end{problem}

See \cite{oberwolfach} for an overview of recent results. To summarize these results, fix a finite morphism $\phi:X\to\PP^1_K$ of degree $n$, and let $p$ be the characteristic of the residue field $k$. Then  we have practical algorithms for solving Problem \ref{prob:ss-red} in each of the following cases:
\begin{itemize}
\item
  The order of the geometric monodromy group of $\phi$ is prime to $p$. For instance, this is the case if $p>n$, or if $p=0$. In this case, the simplest method to compute the semistable reduction of $X$ is typically to use {\em admissible reduction}, see e.g.\ \cite{L-factors}. The case of hyperelliptic curves over $p$-adic number fields with $p\neq 2$ was also considered by Dokchitser et al. (\cite{ddmm}) and Gehrunger and Pink (\cite{gehrunger-pink_p>2}).
\item 
  We have $p=n$. A general approach in this case is given in \cite{OlesThesis}. The case of smooth quartics over $3$-adic number fields has been worked out explicitly in  \cite{OssenQuartic}. The basic method goes back to Lehr and Matignon (\cite{lehr-matignon}). The case of hyperelliptic curves over $2$-adic number fields was recently tackled by  Fiore-Yelton (\cite{fiore-yelton}) and by Gehrunger and Pink (\cite{gehrunger-pink}, \cite{gehrunger}). 
\end{itemize}

The simplest case not covered by any of these methods is that of smooth plane quartics over a field $K$ as above with residue characteristic $p=2$.
The goal of this article is to develop and implement a practical method for computing semistable reduction of smooth plane quartics over $p$-adic number fields,
including the case $p=2$, under the additional assumption of \emph{non-hyperelliptic reduction}.
More precisely, we reduce the computation of the stable model to two explicit steps:
first, computing a GIT-semistable plane model, and second, resolving the cuspidal singularities which may remain in the special fiber of the resulting model.
The first step is treated in the present article (building on \cite{SternWewers,KletusDiss}), while the second step is carried out in the companion paper
\cite{cusp_resolution}.

\smallskip
A key feature of our approach is that it does not start from a chosen cover $X\to\PP^1_K$.
Instead, it exploits projective embeddings (for non-hyperelliptic genus $3$, the canonical embedding $X\hookrightarrow\PP^2_K$) and GIT to control degeneration.
This viewpoint is flexible and suggests a route to explicit semistable reduction for non-hyperelliptic curves of higher genus in small residue characteristic,
via suitable projective models and their GIT geometry.

\subsection*{Plane quartics}

Let $K$ be as before, with algebraically closed residue field $k$ of positive characteristic $p$, and $X\subset \PP^2_K$ be a smooth plane quartic,
\[
X:\; F(x_0,x_1,x_2)=0,
\]
where $F\in K[x_0,x_1,x_2]$ is homogeneous of degree $4$. 
Note that $X$ can be realized as a degree-$4$-cover of $\PP^1_K$, by projection with center any $K$-rational point of $\PP^2_K\backslash X$. So for $p>3$ we can apply the method of admissible reduction, and this is then often the most efficient way to compute the stable reduction of $X$. If $X$ has a $K$-rational point (which can be achieved by passing to a finite extension of $K$) then projecting from this point yields a degree-$3$-map to $\PP^1_K$. For $p=3$ we can then use the results of \cite{OssenQuartic}. But for $p=2$ the available methods, which rely on the single choice of a morphism $\phi:X\to\PP^1_K$, are not applicable.

The method that we propose here works for all primes $p$, and succeeds if and only if the semistable reduction of $X$ is {\em not} hyperelliptic. It relies crucially on the results of \cite{SternWewers} and \cite{hyperelliptic}. The idea is to first, as an intermediate step, find a {\em plane integral model}  which is close to the stable model.
More precisely, we use the results of \cite{SternWewers} to find a finite extension $L/K$ and a plane integral model $\X_0\subset\PP^2_{\OO_L}$ of $X_L$ (the Zariski closure of $X_L\subset\PP^2_L$, after a suitable linear change of coordinates) whose special fiber $\X_{0,s}\subset \PP^2_k$ is {\em semistable in the sense of Geometric Invariant Theory}. We call such a model $\X_0$ a {\em GIT-semistable model}. The philosophy underlying this two-step approach is strongly inspired by the work of Lercier et al. (\cite{LLLR}, \cite{vanBommelDockingLercierLorenzoGarcia2025}), which investigates the reduction of plane quartics via GIT invariants, although the methods and results used in the present article are logically independent.

At this point we encounter a crucial dichotomy. In the first case, the plane model $\X_0$ is actually {\em GIT-stable}. For plane quartics, this means that $\X_{0,s}$ is reduced and has at most nodes and cusps as singularities (plane singularities of type $A_1$ or $A_2$). By the general theory, $\X_0$ is then the unique GIT-semistable model of $X_L$. We show that, after enlarging the extension $L/K$ further, there exists a modification
\[
     \pi:\X\to\X_0
\]
such that $\X$ is the stable model of $X_L$. In fact, $\pi$ is an isomorphism except over the cusps of $\X_{0,s}$, and the exceptional divisors are precisely the {\em $1$-tails} of the stable curve $\X_s$.

The existence of the modification $\pi:\X\to\X_0$ is a consequence of the Semistable Reduction Theorem; in fact, $\X$ is the \emph{stable hull} of $\X_0$,
see \cite{liu2006stable}. What is needed for computations is an explicit construction of $\pi$ and of the finite extension required for its existence.
This second stage---the explicit resolution of the cusps of $\X_{0,s}$ by weighted blow-ups, producing the $1$-tails of the stable special fiber---is the
subject of the companion paper \cite{cusp_resolution}.

\smallskip
\noindent
\emph{Implementation.}
The algorithms of the present article and of \cite{cusp_resolution} are implemented in the SageMath~\cite{SageMath} package
\href{https://github.com/kst3rn/StabilityFunction}{\tt StabilityFunction} (\cite{KletusGitHub}).
The cusp resolution routine is contained in
\href{https://github.com/kst3rn/StabilityFunction/blob/main/semistable_model/curves/cusp_resolution.py}{\tt cusp\_resolution.py};
the main routine for computing stable reduction of smooth plane quartics is contained in
\href{https://github.com/kst3rn/StabilityFunction/blob/main/semistable_model/curves/stable_reduction_of_quartics.py}{\tt stable\_reduction\_of\_quartics.py}.

\medskip
In the second case, $\X_0$ is GIT-semistable but not GIT-stable and then our method fails. Here we show that this happens if and only if $X$ has {\em hyperelliptic reduction}, i.e.\ the special fiber $\X_s$ of the stable model lies, as a point on $\overline{\M}_3$, in the closure of the hyperelliptic locus. In this case, the relationship between GIT-semistable plane models and the stable model is considerably more subtle, and it is unclear how our basic strategy can be adapted. We hope to consider this problem in a future article.

\vspace{4ex}
The article is organized as follows.
In Section \ref{sec:stable_curves} we recall the basic notions of stable curves of genus $3$ and review the classification of their combinatorial types,
with particular emphasis on hyperelliptic versus non-hyperelliptic reduction.
In Section \ref{sec:plane_models} we study plane models of smooth quartics and recall the relevant GIT-stability criteria, leading to the notions of GIT-stable and GIT-semistable plane models.
We also recall the results of \cite{SternWewers}, which lead to an algorithm for computing GIT-semistable models of plane curves, and their implementation
(\cite{KletusDiss}, \cite{KletusGitHub}).

In Section \ref{sec:contraction} we establish the precise relationship between Deligne--Mumford stable reduction and GIT-stability for plane quartics, and prove the main theoretical result
explaining the dichotomy between non-hyperelliptic versus hyperelliptic reduction. More precisely, we show that non-hyperelliptic stable reduction is equivalent to the
existence of a GIT-stable plane model, while hyperelliptic reduction necessarily leads to strictly semistable behavior in the GIT sense. The proof relies on a detailed
analysis of dualizing sheaves, canonical maps of singular curves, and the contraction of $1$-tails, and isolates the exact geometric mechanism behind this dichotomy.

Finally, in \S\ref{subsec:explicit_cusp_resolution} we explain how the results of this article connect to the explicit construction of the stable model in the
non-hyperelliptic case: starting from a GIT-stable plane model, one resolves the cusps of the special fiber by an explicit weighted blow-up procedure.
The details of this resolution step are given in the companion paper \cite{cusp_resolution}.
We end the paper with explicit examples in Section~\ref{sec:examples}, including a brief demonstration of our implementation \cite{KletusGitHub}.

The theoretical foundations of this article were first developed in the first named author's master's thesis~\cite{MaxMaster}; the results on hyperelliptic stable curves are cited here in their refined form in~\cite{hyperelliptic}.

% !TeX root = reduction_of_plane_quartics.tex

\section{Stable curves of genus \texorpdfstring{$3$}{3}} \label{sec:stable_curves}
 
In this section we recall the basic terminology for stable pointed curves and summarize results from \cite{hyperelliptic}. Our main focus is to clarify the notion of a {\em hyperelliptic} stable curve and to revisit the classification of stable curves of genus 
$g=3$ from this perspective.

\subsection{Basic definitions}

In this section we fix an algebraically closed field $k$.

\begin{definition} \label{def:semistable_curve}
    \begin{enumerate}
        \item
        An algebraic curve $C$ over $k$ is called {\em semistable} if it is connected, projective, reduced, and its only singularities are ordinary double points.
        \item
        Let $C$ be a semistable curve over $k$. A {\em marking} of $C$ is a finite set $S$ of closed, smooth points of $C$. The pair $(C,S)$ is called a {\em marked semistable curve}, or an {\em $r$-marked semistable curve} if $r=\abs{S}$. We do not distinguish between a semistable curve and a $0$-marked semistable curve.
        \item
        Let $(C,S)$ be a marked semistable curve and $C_1,\ldots,C_n$ its irreducible components. The {\em special points} on $C_i$ are the elements of $S\cap C_i$ and the points where $C_i$ meets one of the other components.\footnote{Note that nodes where a component intersects itself do not count as special points.} Let $g_i$ denote the arithmetic genus of $C_i$ and $r_i$ the number of special points on $C_i$. We call $(C,S)$ {\em stably marked} if  $2g_i-2+r_i >0$, for all $i$.
    \end{enumerate}
\end{definition}

\begin{remark}
    Note that the number $2g_i-2+r_i$ is precisely the degree of the restriction of the log-canonical sheaf $\omega_{C/k}(S)$ to the component $C_i$. Consequently, the condition of being stably marked is equivalent to $\omega_{C/k}(S)$ being ample. It is also equivalent to the finiteness of the automorphism group $\Aut(C,S)$ (see \cite[Chapter X, 3]{ACG}).
\end{remark}

Given a semistable marked curve $(C,S)$, we define its {\em dual graph} $\Gamma=\Gamma(C,S)$ as the following marked, weighted graph:
\begin{itemize}
    \item
    The set of vertices is the set $V=\{C_1,\ldots,C_n\}$ of irreducible components.
    \item
    To each node $x\in C$ we associate an edge $e_x$ connecting the two components $C_i,C_j$ which meet in $x$ ($i=j$ is allowed).
    \item
    To each marked point $p\in S$ we associate a half edge $e_p$ which is connected to the component on which $p$ lies.
    \item
    The weight of a vertex is the geometric genus $g_i$ of the component $C_i$ to which it corresponds.
\end{itemize}

Then the arithmetic genus of $C$ is given by the following formula:
\[
g(C) = \sum_i g_i + h^1(\Gamma),
\]
where $h^1(\Gamma)$ denotes the first Betti number of the graph. The isomorphism class of this dual graph is called the \emph{combinatorial type} of the curve $(C,S)$.

\subsection{Hyperelliptic stable curves}

\begin{definition} \label{def:hyperelliptic}
    A stably marked curve $(C,S)$ is called {\em hyperelliptic} if there exists an involution $\sigma:C\to C$, $\sigma^2=\Id_C$, with the following properties:
    \begin{enumerate}
        \item
        $\sigma$ fixes every point in $S$.
        \item
        If $\Char(k)\neq 2$, then $\sigma$ has only isolated fixed points.
        \item
        The quotient curve $C/\gen{\sigma}$ has arithmetic genus zero.
    \end{enumerate}
\end{definition}

The involution $\sigma$, if it exists, is unique (see \cite[Theorem 4.5]{hyperelliptic}) and is called the \emph{hyperelliptic involution} of $(C,S)$. This definition is motivated by the following theorem.

\begin{theorem}[{{\cite[Theorem 5.5]{hyperelliptic}}}] \label{thm:hyperelliptic_locus}
    Let $C$ be a stable curve over $k$ of genus $g\geq 2$. Then $C$ is hyperelliptic if and only if the corresponding point on $\overline{\M}_g$ lies in the closure of the hyperelliptic locus $\Hc_g\subset\M_g$.
\end{theorem}

To discuss hyperelliptic stable curves, it is convenient to recall the following combinatorial terminology from \cite{hyperelliptic}.

\begin{definition} \label{def:comb_term}
Let $C$ be a semistable curve over $k$.
\begin{enumerate}
    \item A node $p$ is \emph{separating} if the partial normalization of $C$ at $p$ disconnects the curve. (Equivalently, the corresponding edge in the dual graph is a bridge.)
    \item A pair of distinct nodes $\{p, q\}$ is a \emph{separating pair} if the partial normalization of $C$ at $\{p, q\}$ disconnects the curve, but the partial normalization at either $p$ or $q$ individually does not. (Equivalently, the corresponding edges form a minimal 2-edge-cut in the dual graph.)
    \item The curve $C$ is \emph{inseparable}\footnote{This terminology is inspired by \cite{ran2014canonical}. This geometric notion should not be confused with the algebraic concept of an inseparable field extension or morphism.} if it has no separating nodes, and is \emph{separable} otherwise. It is \emph{2-inseparable} if it is inseparable and contains no separating pairs. (Combinatorially, these conditions correspond to the dual graph being 2-edge-connected and 3-edge-connected, respectively.)
\end{enumerate}
\end{definition}

\subsection{Classification of stable curves of genus \texorpdfstring{$3$}{3}}

We now restrict our attention to stable curves of genus $g=3$ and apply the results of \cite{hyperelliptic} to make the condition of hyperellipticity fully explicit in this case. It is well known that there are precisely $42$ different combinatorial types of stable curves of genus $g=3$. We refine this classification using the notion of the {\em core} and {\em 1-tails}, shedding some new light on the distinction between hyperelliptic and non-hyperelliptic curves.

\begin{definition}
    Let $C$ be a stable curve over $k$, and let $E$ be an irreducible component of $C$.
    \begin{enumerate}
        \item $E$ is called a \emph{1-tail} if it is attached to the rest of the curve at a single node and has an arithmetic genus of $1$.
        \item An \emph{elliptic tail} is a 1-tail that is smooth (i.e., has geometric genus $1$).
        \item A \emph{pigtail}\footnote{This is an actual term used in the literature; see, for instance, \cite[p. 122, 175]{harris1998moduli}.} is a 1-tail that has a single node (i.e., has geometric genus $0$).
    \end{enumerate}
    The node at which a 1-tail $E$ is attached is called the \emph{attachment point} of $E$.
\end{definition}

For stable curves of genus $g=3$, a simple argument using the additivity of the arithmetic genus shows that separating nodes are in one-to-one correspondence with 1-tails. More precisely:

\begin{lemma}[{{\cite[Lemma 2.35]{MaxMaster}}}] \label{lem:core_tail_lemma}
    Let $C$ be a stable curve over $k$ of genus $3$, and let $r$ be the number of 1-tails of $C$. Then $0 \leq r \leq 3$. The curve has exactly $r$ separating nodes, which are precisely the attachment points of the 1-tails. The partial normalization of $C$ at these nodes results in the disjoint union of $r$ 1-tails and an inseparable curve $C_c$ of genus $g(C_c) = 3 - r$.
\end{lemma}

This observation motivates the following definition.

\begin{definition}
    For a stable curve $C$ over $k$ of genus $3$, we call the curve $C_c$ from \thref{lem:core_tail_lemma} the \emph{core} of $C$.
\end{definition}

We use the naming convention of \cite{octads}:
\begin{itemize}
    \item A number (\texttt{0}, \texttt{1}, \texttt{2}, or \texttt{3}) indicates an irreducible component of that \emph{geometric} genus.
    \item The letter \texttt{n} appended to a genus number (e.g., \texttt{0n}) signifies a self-intersection (a node) on that component.
    \item The letter \texttt{e} denotes an \emph{elliptic tail} attached to the previously mentioned component.
    \item The letter \texttt{m} denotes a \emph{pigtail} (from ``multiplicative reduction'') attached to the previously mentioned component.
    \item Special characters describe intersections between two components: \texttt{=} for two intersection points, \texttt{-{}-{}-} for three, and \texttt{-{}-{}-{}-} for four.
    \item \texttt{Z} is a convenient shorthand for the common configuration \texttt{0=0}, a binary curve with two nodes.
    \item Two exceptional graphs have special names, following the notation in \cite{Ciani1}: \texttt{CAVE} and \texttt{BRAID}.
\end{itemize}

Using the terminology of 1-tails and the core, the classification of all 42 combinatorial types of stable curves of genus $3$ can be formulated as follows.

\begin{proposition}[Classification of stable curves of genus $3$] \label{prop:classification_g=3}
  There are exactly $42$ combinatorial types of stable curves of genus $3$. They may be classified according to the properties of their core, as follows:
\begin{description}
    \item Core is 2-inseparable (27 cases)
    \begin{description}
        \item Core is irreducible (20 cases)
        \par
        \vspace{0.5em}
        \centerline{%
        \begin{tabular}{Sc|Sc|Sc|Sc}
        \small0 1-tails (4) & \small1 1-tail (6) & \small2 1-tails (6) & \small3 1-tails (4) \\
        \hline
        \begin{tabular}{@{}cc@{}}
            \HobbyCurve{3} & \HobbyCurve{2n} \\[-0.80em]
            \footnotesize\texttt{3} & \footnotesize\texttt{2n} \\[-0.25em]
            \HobbyCurve{1nn} & \HobbyCurve{0nnn} \\[-0.80em]
            \footnotesize\texttt{1nn} & \footnotesize\texttt{0nnn}
        \end{tabular}
        &
        \begin{tabular}{@{}cc@{}}
            \HobbyCurve{2e} & \HobbyCurve{2m} \\[-0.80em]
            \footnotesize\texttt{2e} & \footnotesize\texttt{2m} \\[-0.25em]
            \HobbyCurve{1ne} & \HobbyCurve{1nm} \\[-0.80em]
            \footnotesize\texttt{1ne} & \footnotesize\texttt{1nm} \\[-0.25em]
            \HobbyCurve{0nne} & \HobbyCurve{0nnm} \\[-0.80em]
            \footnotesize\texttt{0nne} & \footnotesize\texttt{0nnm}
        \end{tabular}
        &
        \begin{tabular}{@{}cc@{}}
            \HobbyCurve{1ee} & \HobbyCurve{1me} \\[-0.80em]
            \footnotesize\texttt{1ee} & \footnotesize\texttt{1me} \\[-0.25em]
            \HobbyCurve{1mm} & \HobbyCurve{0nee} \\[-0.80em]
            \footnotesize\texttt{1mm} & \footnotesize\texttt{0nee} \\[-0.25em]
            \HobbyCurve{0nme} & \HobbyCurve{0nmm} \\[-0.80em]
            \footnotesize\texttt{0nme} & \footnotesize\texttt{0nmm}
        \end{tabular}
        &
        \begin{tabular}{@{}cc@{}}
            \HobbyCurve{0eee} & \HobbyCurve{0mee} \\[-0.80em]
            \footnotesize\texttt{0eee} & \footnotesize\texttt{0mee} \\[-0.25em]
            \HobbyCurve{0mme} & \HobbyCurve{0mmm} \\[-0.80em]
            \footnotesize\texttt{0mme} & \footnotesize\texttt{0mmm}
        \end{tabular}
        \end{tabular}
        }

        \item Core is reducible (7 cases)
            \par
            \centerline{%
            \renewcommand{\arraystretch}{1.2}
            \begin{tabular}{Sc|Sc}
            \small0 1-tails (5) & \small1 1-tail (2) \\
            \hline
            \begin{tabular}{@{}c@{}}
                \setlength{\tabcolsep}{1pt}
                \begin{tabular}{ccccc}
                \HobbyCurve{1---0} & \HobbyCurve{0---0n} & \HobbyCurve{0----0} & \HobbyCurve{CAVE} & \HobbyCurve{BRAID} \\[-0.60em]
                \footnotesize\texttt{1-{}-{}-0} & \footnotesize\texttt{0-{}-{}-0n} & \footnotesize\texttt{0-{}-{}-{}-0} & \footnotesize\texttt{CAVE} & \footnotesize\texttt{BRAID}
                \end{tabular}
            \end{tabular}
            &
            \begin{tabular}{@{}c@{}}
                \begin{tabular}{cc}
                \HobbyCurve{0---0e} & \HobbyCurve{0---0m} \\[-0.60em]
                \footnotesize\texttt{0-{}-{}-0e} & \footnotesize\texttt{0-{}-{}-0m}
                \end{tabular}
            \end{tabular}
            \end{tabular}
            }
    \end{description}

    \item Core is 2-separable (15 cases)
    \begin{description}
        \item Core has 2 components (10 cases)
            \par
            \vspace{0.5em}
            \centerline{%
            \renewcommand{\arraystretch}{1.2}
            \begin{tabular}{Sc|Sc|Sc}
            \small0 1-tails (3) & \small1 1-tail (4) & \small2 1-tails (3) \\
            \hline
            \begin{tabular}{@{}c@{}}
                \setlength{\tabcolsep}{3pt}
                \begin{tabular}{ccc}
                \HobbyCurve{1=1} & \HobbyCurve{1=0n} & \HobbyCurve{0n=0n} \\[-0.50em]
                \footnotesize\texttt{1=1} & \footnotesize\texttt{1=0n} & \footnotesize\texttt{0n=0n}
                \end{tabular}
            \end{tabular}
            &
            \begin{tabular}{@{}c@{}}
                \setlength{\tabcolsep}{2pt}
                \begin{tabular}{cccc}
                \HobbyCurve{1=0e} & \HobbyCurve{1=0m} & \HobbyCurve{0n=0e} & \HobbyCurve{0n=0m} \\[-0.50em]
                \footnotesize\texttt{1=0e} & \footnotesize\texttt{1=0m} & \footnotesize\texttt{0n=0e} & \footnotesize\texttt{0n=0m}
                \end{tabular}
            \end{tabular}
            &
            \begin{tabular}{@{}c@{}}
                \setlength{\tabcolsep}{3pt}
                \begin{tabular}{ccc}
                \HobbyCurve{0e=0e} & \HobbyCurve{0m=0e} & \HobbyCurve{0m=0m} \\[-0.50em]
                \footnotesize\texttt{0e=0e} & \footnotesize\texttt{0m=0e} & \footnotesize\texttt{0m=0m}
                \end{tabular}
            \end{tabular}
            \end{tabular}
            }

        \item Core has at least 3 components (5 cases)
            \par
            \vspace{0.5em}
            \centerline{%
            \renewcommand{\arraystretch}{1.2}
            \begin{tabular}{Sc|Sc}
            \small0 1-tails (3) & \small1 1-tail (2) \\
            \hline
            \begin{tabular}{@{}c@{}}
                \setlength{\tabcolsep}{3pt}
                \begin{tabular}{ccc}
                \HobbyCurve{Z=1} & \HobbyCurve{Z=0n} & \HobbyCurve{Z=Z} \\[-0.50em]
                \footnotesize\texttt{Z=1} & \footnotesize\texttt{Z=0n} & \footnotesize\texttt{Z=Z}
                \end{tabular}
            \end{tabular}
            &
            \begin{tabular}{@{}c@{}}
                \begin{tabular}{cc}
                \HobbyCurve{Z=0e} & \HobbyCurve{Z=0m} \\[-0.50em]
                \footnotesize\texttt{Z=0e} & \footnotesize\texttt{Z=0m}
                \end{tabular}
            \end{tabular}
            \end{tabular}
            }
    \end{description}
\end{description}
The geometric genus of each component ($0$, $1$, $2$, or $3$) is indicated by its line thickness, from thinnest to thickest.
\end{proposition}

\begin{proof}
The diagrams in the classification above are adapted from \cite[Figure 2.2]{octads}, with slight modifications made using the Hobby-Editor tool \cite{hobbyeditor}. The enumeration of all stable types of a genus $g$ curve is a well-known combinatorial problem (see for example \cite[Theorem 2.2.12]{chan2013tropical}). A direct proof for $g=3$ using the core and 1-tail decomposition is given in \cite[Proposition 2.38]{MaxMaster}.
\end{proof}

We can now leverage the results from \cite{hyperelliptic} to describe the precise conditions under which each combinatorial type is hyperelliptic. For a compact notation, we introduce some terminology to refer to similar cases at once. 
We use the wildcard \texttt{$\ast$=$\ast$} for the 15 combinatorial types with
$2$-separable core. Furthermore, since it is convenient not to distinguish whether a 1-tail is an elliptic tail (\texttt{e}) or a pigtail (\texttt{m}), we replace both with the placeholder \texttt{p}. For example, \texttt{2p} stands for either of the combinatorial types \texttt{2e} or \texttt{2m}. For a curve with $2$-inseparable core, this notation may also be read as notation for the corresponding pointed core
$(C_c,\{p_1,\ldots,p_r\})$.

\begin{theorem} \label{thm:hyperelliptic_classification_g3}
	Let $C$ be a stable curve of genus $3$. The possibility of $C$ being
	hyperelliptic depends on its combinatorial type and on the characteristic of
	the base field $k$ as follows:
	\begin{center}
		\normalfont
		\renewcommand{\arraystretch}{1.3}
		\begin{tabular}{@{}ll@{}}
			\toprule
			\textbf{Hyperellipticity} & \textbf{Combinatorial Types} \\
			\midrule
			Always hyperelliptic & \texttt{$\ast$=$\ast$} \\
			Never hyperelliptic & \texttt{1---0}, \texttt{0---0n}, \texttt{CAVE}, \texttt{BRAID}, \texttt{0---0p} \\
			Hyperelliptic iff $\Char(k) = 2$ & \texttt{0ppp} \\
			Both possible (never if $\Char(k)=2$) & \texttt{0npp} \\
			Both possible & \texttt{3}, \texttt{2n}, \texttt{1nn}, \texttt{0nnn}, \texttt{2p}, \texttt{1np}, \\
			& \texttt{0nnp}, \texttt{1pp}, \texttt{0----0} \\
			\bottomrule
		\end{tabular}
	\end{center}
	In particular, if $C$ is non-hyperelliptic, then its core $C_c$ is
	$2$-inseparable.
\end{theorem}

\begin{proof}
	In genus $3$, a stable curve has at most one separating pair, and hence is of semicompact type in the sense of \cite[Definition~2.3(v)]{hyperelliptic}, which is a necessary property for a stable curve to be hyperelliptic. In this case, \cite[Theorem~4.6]{hyperelliptic} tells exactly when a curve is hyperelliptic; for us, \cite[Lemma~4.3]{hyperelliptic} is useful in the $2$-separable core case.
	
	If the core is $2$-separable, the decomposition of the curve in the sense of \cite[Lemma~4.3]{hyperelliptic} at the separating pair gives two contracted components of exactly genus $2$. Since every stable curve of genus $2$ is hyperelliptic, it follows by the lemma that the curve itself is hyperelliptic. Equivalently, all curves of type \texttt{$\ast$=$\ast$} are hyperelliptic. This proves the first row and, by contraposition, the final assertion.
	
	It remains to consider the case where the core is $2$-inseparable. By
	Lemma~\ref{lem:core_tail_lemma}, the $2$-inseparable components (in the sense of \cite[Theorem~4.6]{hyperelliptic}) are the pointed
	core $(C_c,\{p_1,\ldots,p_r\})$ and the $r$ $1$-tails. The $1$-tails are
	hyperelliptic as pointed genus-$1$ components, so
	\cite[Theorem~4.6]{hyperelliptic} reduces the question to the pointed core. The
	table is then the genus-$3$ specialization of
	\cite[Theorem~4.6 and Proposition~4.7]{hyperelliptic}. The exceptional
	\texttt{0ppp} row is explained by \cite[Example~4.9]{hyperelliptic}, and the
	characteristic-$2$ exclusion in the \texttt{0npp} row by
	\cite[Example~4.10]{hyperelliptic}. The rows marked ``both possible'' depend on
	additional geometric conditions on the pointed core, not only on the
	combinatorial type.
\end{proof}

Below in \S\ref{subsec:construction_of_GIT_model}, rather than using these
additional geometric conditions case by case, we will use the pinching criterion
of \cite[Definition~4.12, Theorem~4.13]{hyperelliptic}.

% !TeX root = reduction_of_plane_quartics.tex

\section{Plane models of smooth quartics} \label{sec:plane_models}

In this section we discuss a different notion of (semi)stability of curves. It only applies to plane curves and is directly derived from Mumford's Geometric Invariant Theory (GIT). 

\subsection{Plane curves} \label{subsec:plane_curves}

Let $k$ be a field. By a {\em plane curve} of degree $d$ over $k$ we mean a hypersurface
\[
     X = V_+(F)\subset\PP^2_k,
\]
where $F\in W:=k[x,y,z]_d$, $F\neq 0$, is a homogeneous ternary form of degree $d$. We consider such plane curves up to projective equivalence. More formally, we identify the isomorphism class of the plane curve $X$ with the orbit of the homothety class $[F]\in\PP(W)$ under the natural action of $\PGL_3(k)$.

\emph{Geometric invariant theory} (GIT) distinguishes among orbits of $\PGL_3(k)$ on $\PP(W)$ according to their stability properties (see \cite{MumfordGIT}). An orbit may be stable, semistable, or unstable, where \emph{stable} implies \emph{semistable}, and \emph{unstable} simply means \emph{not semistable}. We also use \emph{strictly semistable} meaning semistable but not stable. For our purposes we do not need the general definitions, since in the case of plane quartics the distinction can be expressed by a simple geometric criterion (see Proposition~\ref{prop:semistable} at the end of this subsection). To avoid confusion with the notion of \emph{stable curves} from \S\ref{sec:stable_curves}, we shall say that a plane curve $X=V_+(F)\subset\PP^2_k$ is \emph{GIT-stable} (resp.\ \emph{GIT-semistable}, etc.) if the orbit of $[F]$ has the corresponding property.

\vspace{2ex}
Before stating the geometric criterion for GIT-(semi)stability we recall some basic terminology for plane curve singularities. Let \(X\subset\PP^2_k\) be a plane curve and \(P\in X\) a closed point. We say that \(P\) is an {\em \(n\)-fold point} of \(X\) if the multiplicity of \(X\) at \(P\) equals \(n\). In particular, an {\em ordinary double point} (or {\em node}) is a double point whose tangent cone consists of two distinct lines.

Since cusps will play an important role below, we record a precise definition, valid for any algebraic curve (not necessarily a plane curve).

\begin{definition}
    Let \(X\) be an algebraic curve over \(k\) (a scheme of finite type and pure dimension one over \(k\)), and \(P\in X\) a closed point. We call \(P\) a {\em cusp} if:
    \begin{enumerate}
        \item \(\widehat{\OO}_{X,P}\) is a domain (equivalently, \(X\) is analytically irreducible at \(P\); in particular, unibranch and generically reduced at \(P\));
        \item the {\em value semigroup} at \(P\) is \(\langle 2,3\rangle=\{0,2,3,4,5,\ldots\}\).
    \end{enumerate}
    Here the value semigroup is
    \[
    \Gamma_P:=\{\nu(f):0\neq f\in\widehat{\OO}_{X,P}\}\subset\NN_0,
    \]
    where $\nu$ is the discrete valuation on the function field of $X$ corresponding to the unique branch at $P$.
\end{definition}

It is well known that \(P\) is a cusp in this sense if and only if
\[
\widehat{\OO}_{X,P}\;\cong\;k[[x,y]]/(y^2-x^3),
\]
i.e. the singularity is analytically isomorphic to the plane curve \(y^2=x^3\) (valid in arbitrary characteristic); see, e.g., \cite{Campillo,Wall}. Within the ADE classification for plane curve singularities, nodes are precisely type \(A_1\) and cusps are type \(A_2\); see \cite{GreuelKroening}.

\vspace{2ex}

For our main result we will need to know that we can bring any cusp into a certain {\em normal form}. For a more precise statement, we use the following notation. Given a plane curve $X=V_+(F)\subset\PP^2_k$ of degree $d$, we write the equation $F\in k[x,y,z]_d$ as
\[
F = \sum_{i+j\leq d} \, a_{i,j} x^iy^jz^{d-i-j},
\]
with $a_{i,j}\in k$. Of course, this presentation depends on the chosen coordinate system $(x,y,z)$ which, as explained above, we are allowed to change by any projective linear transformation given by an element of $\PGL_3(k)$.

\begin{lemma} \label{lem:cusp_normal_form}
    Let $X\subset\PP^2_k$ be a plane curve of degree $d\geq 3$ and $P\in X$ a cusp. Then there exists a coordinate system $(x,y,z)$ for $\PP^2_k$ such that $P=(0:0:1)$ and
    \begin{equation}
        a_{i,j} = 0 \quad\text{for $2i+3j<6$}, \quad a_{3,0}=-1,\; a_{0,2}=1.
    \end{equation}
\end{lemma}

\begin{proof}
This is standard, see e.g. \cite[Exercise 3.22]{FultonAlgebraicCurves}.
\end{proof}

Finally, we state the geometric criterion for GIT-(semi)stability:

\begin{proposition} \label{prop:semistable}
  Let $X\subset\PP^2_k$ be a plane quartic.
  \begin{enumerate}[(i)]
  \item
    $X$ is GIT-semistable if and only if $X$ has no $n$-fold point with $n\geq 3$, and is not the union of a cubic with one of its inflectional tangents.
  \item
    $X$ is GIT-stable if and only if it is reduced and has at most ordinary double points or cusps as singularities.
 \end{enumerate}
\end{proposition}

\begin{proof}
See \cite[Ch.~4, §2, Prop.~4.2]{MumfordGIT}.
\emph{A note on terminology:} In \emph{loc.\ cit.} “triple point'' means a point of multiplicity $\geq 3$, and the word “cusp’’ means any non-ordinary double point.
For plane quartics the non-ordinary double points are precisely the \(A_2\) cusp (our “cusp’’) and the \(A_3\) tacnode.
Translating the wording in \emph{loc.\ cit.} via this convention yields: unstable \(\Leftrightarrow\) triple point or a cubic plus one of its inflectional tangents; tacnode (and the double conic) give strictly semistable; stable \(\Leftrightarrow\) reduced with only nodes and \(A_2\) cusps. This is exactly the formulation in the proposition.
\end{proof}

\emph{Cf. also} \cite[§2]{vanBommelDockingLercierLorenzoGarcia2025},
which restates Mumford’s quartic GIT classification in modern invariant language; in particular,
tacnodal quartics are strictly semistable, while quartics with a triple point or a cubic together with one of its inflectional tangents are unstable.

We stratify GIT-stable plane quartics $X \subset \PP^2_k$ by their combinatorial type. Let $\nu\colon \tilde{X} \to X$ be the partial normalization of $X$ at its cusps, and let $S$ be the set of preimages of these cusps. Because a GIT-stable quartic has at most ordinary double points and cusps as singularities (Proposition~\ref{prop:semistable}), the partial normalization $\tilde{X}$ has only ordinary double points and is therefore a semistable curve. We define the \emph{combinatorial type} of the GIT-stable quartic $X$ as the combinatorial type of the stably marked semistable curve $(\tilde{X}, S)$.

To describe these combinatorial types compactly, we adopt a naming convention similar to the one used for stable curves of genus $3$. We indicate the geometric genus of a component with a number (\texttt{0}, \texttt{1}, \texttt{2}, or \texttt{3}) and append the letter \texttt{n} for each node (self-intersection). For the preimages of the cusps, we now append the letter \texttt{c}, which acts analogously to \texttt{n}. For instance, \texttt{1nc} denotes an irreducible curve of geometric genus $1$ possessing exactly one node and one cusp. As before, two exceptional reducible configurations are denoted by the special names \texttt{CAVE} and \texttt{BRAID}.

\begin{theorem} \label{thm:git_stable_types}
    There are exactly $16$ combinatorial types of GIT-stable plane quartics. They can be classified by the number of irreducible components and the number of cusps, as follows:

    \begin{description}
        \item Irreducible (10 cases)
        \par
        \vspace{0.5em}
        \centerline{%
        \begin{tabular}{Sc|Sc|Sc|Sc}
        \small 0 cusps (4) & \small 1 cusp (3) & \small 2 cusps (2) & \small 3 cusps (1) \\
        \hline
        \begin{tabular}{@{}cc@{}}
            \HobbyCurve{3} & \HobbyCurve{2n} \\[-0.80em]
            \footnotesize\texttt{3} & \footnotesize\texttt{2n} \\[-0.25em]
            \HobbyCurve{1nn} & \HobbyCurve{0nnn} \\[-0.80em]
            \footnotesize\texttt{1nn} & \footnotesize\texttt{0nnn}
        \end{tabular}
        &
        \begin{tabular}{@{}cc@{}}
            \HobbyCurve{2c} & \HobbyCurve{1nc} \\[-0.80em]
            \footnotesize\texttt{2c} & \footnotesize\texttt{1nc} \\[-0.25em]
            \multicolumn{2}{c}{\HobbyCurve{0nnc}} \\[-0.80em]
            \multicolumn{2}{c}{\footnotesize\texttt{0nnc}}
        \end{tabular}
        &
        \begin{tabular}{@{}c@{}}
            \HobbyCurve{1cc} \\[-0.80em]
            \footnotesize\texttt{1cc} \\[-0.25em]
            \HobbyCurve{0ncc} \\[-0.80em]
            \footnotesize\texttt{0ncc}
        \end{tabular}
        &
        \begin{tabular}{@{}c@{}}
            \HobbyCurve{0ccc} \\[-0.80em]
            \footnotesize\texttt{0ccc}
        \end{tabular}
        \end{tabular}
        }

        \item Reducible (6 cases)
        \par
        \vspace{0.5em}
        \centerline{%
        \renewcommand{\arraystretch}{1.2}
        \begin{tabular}{Sc|Sc}
        \small 0 cusps (5) & \small 1 cusp (1) \\
        \hline
        \begin{tabular}{@{}c@{}}
            \setlength{\tabcolsep}{1pt}
            \begin{tabular}{ccccc}
            \HobbyCurve{1---0} & \HobbyCurve{0---0n} & \HobbyCurve{0----0} & \HobbyCurve{CAVE} & \HobbyCurve{BRAID} \\[-0.60em]
            \footnotesize\texttt{1-{}-{}-0} & \footnotesize\texttt{0-{}-{}-0n} & \footnotesize\texttt{0-{}-{}-{}-0} & \footnotesize\texttt{CAVE} & \footnotesize\texttt{BRAID}
            \end{tabular}
        \end{tabular}
        &
        \begin{tabular}{@{}c@{}}
            \begin{tabular}{c}
            \HobbyCurve{0---0c} \\[-0.60em]
            \footnotesize\texttt{0-{}-{}-0c}
            \end{tabular}
        \end{tabular}
        \end{tabular}
        }
    \end{description}
\end{theorem}

\begin{proof}
    This classification is elementary and can be verified using combinatorial degree-genus arguments similar to those in \cite[Lemma 3.49 a)]{MaxMaster}.
\end{proof}

\begin{remark} \label{rem:tricuspidal_char2}
    We note that the combinatorial type \texttt{0ccc}, which corresponds to the tricuspidal quartic, can only occur if $\Char(k) \neq 2$ (cf.\ \cite[Remark A.2]{MaxMaster}). All other $15$ combinatorial types can be realized in any characteristic.
\end{remark}

\subsection{Plane models of curves} \label{subsec:plane_models}

From now on and for the rest of this article, $K$ will denote a field which is complete with respect to a discrete valuation $v_K$. We let $\OO_K\subset K$ denote the valuation ring, $\pi$ some uniformizer and $k:=\OO_K/(\pi)$ the residue field of $v_K$. We also assume that $k$ is algebraically closed.

Let $X$ be a smooth projective curve over $K$. By a {\em model} of $X$ we mean a flat and proper $\OO_K$-scheme $\X$ with generic fiber $X$.\footnote{In most related articles we also demand that models be normal schemes. We do not include this condition here because plane models are often not normal.} Then $\X_s:=\X\otimes_{\OO_K} k$ is called the {\em special fiber} of the model $\X$; it is a proper algebraic curve. Given two models $\X$ and $\X'$ of $X$ we say that $\X'$ {\em dominates} $\X$ if the identity on $X$ extends to a morphism $\X'\to\X$. We say that $\X'$ and $\X$ are {\em isomorphic} if this morphism is an isomorphism\footnote{Equivalently, $\X$ and $\X'$ dominate each other.}.

We fix a plane curve $X=V_+(F)\subset\PP^2_K$ of degree $d\geq 3$. We assume that $X$ is smooth. By \cite[Proposition 4.2]{MumfordGIT} this implies that $X$ is GIT-stable (for $d=4$ this is also a consequence of Proposition \ref{prop:semistable} (ii)).

One way to construct a model of a plane curve is as follows. Let $\PP_{\OO_K}^2=\Proj(\OO_K[x,y,z])$ denote the projective plane over $\Spec\OO_K$ and let $\X:=\overline{X}\subset\PP^2_{\OO_K}$ be the Zariski closure of $X$. Then $\X$ is a model of $X$. Explicitly, given a plane curve $X=V_+(F)$ of degree $d$, we may assume, after multiplying $F$ with a suitable constant, that $F\in\OO_K[x,y,z]_d$ is primitive. Then the model $\X$ is simply the projective scheme over $\OO_K$ defined by $F$, i.e.\ $\X:=V_+(F)\subset\PP^2_{\OO_K}$. In particular, the special fiber $\X_s$ is the plane curve over $k$ defined by the reduction $\bar{F}\in k[x,y,z]_d$ of $F$. A model $\X=V_+(F)$ as defined above is called a {\em plane model} of $X$.

The definition of the plane model $\X$ clearly depends on the chosen embedding $X\subset\PP^2_K$. If we replace this embedding by its composition with the projective automorphism of $\PP^2_K$ induced by an element $g\in\GL_3(K)$ we obtain another plane model of $X$ which we call $\lexp{g}{\X}$. In more concrete terms this means that we make a linear change of variables
\begin{equation}
    (x,y,z) \;\leftarrow\; (x,y,z)\cdot T,
\end{equation}
where $T\in\GL_3(K)$ represents the contragredient of $g$.
Then the model $\lexp{g}{\X}$ is defined by the equation
\[
    \tilde{F}:=\pi^{-m}\cdot\lexp{T}{F}\in\OO_K[x,y,z]_d,
\]
where
\[
    \lexp{T}{F} := F((x,y,z)\cdot T) \in K[x,y,z]_d
\]
and $m$ is chosen such that $\tilde{F}$ is integral and primitive. It is clear that $\lexp{g}{\X}$ is isomorphic to $\X$ if and only if $g\in K^\times\cdot\GL_3(\OO_K)$. 

\begin{definition}[{{\cite[Definition 1.1]{SternWewers}}}]
  A plane model $\X\subset\PP^2_{\OO_K}$ of $X$ is called \emph{GIT-(semi)stable} if its special fiber $\X_s\subset\PP^2_k$ is GIT-(semi)stable.
\end{definition}

We have an analogue of the Semistable Reduction Theorem for GIT-semistable models:

\begin{theorem}[{{\cite[Theorem 1.3]{SternWewers}}}] \label{thm:GIT-ss-reduction}
   Let $X\subset \PP^2_K$ be a smooth plane curve of degree $d\geq 3$ over $K$. Then there exists a finite extension $L/K$ and a GIT-semistable plane model $\X$ of $X_L$. Moreover, if this model is GIT-stable, then it is the unique GIT-semistable model, up to isomorphism.
\end{theorem}

\begin{proof}
See e.g.\ \cite[Lemma 5.3]{MumfordSPV}, \cite[Proposition 2]{Burnol}, or \cite[Theorem 3.3]{LLLR}.
\end{proof}
	
The proofs of Theorem \ref{thm:GIT-ss-reduction} that we cited all use Geometric Invariant Theory and ultimately rely on the properness of the Proj of a ring of invariants. As the ring of invariants is hard to make explicit, it is not clear how to turn such a proof into a practical algorithm for computing a GIT-semistable model. 

In \cite[Theorem 1.8 and \S 5]{SternWewers} we gave an alternative proof of Theorem \ref{thm:GIT-ss-reduction} which uses different techniques more amenable to explicit computations. Based on this new approach, the second author has developed and implemented an algorithm which computes, on input a stable plane curve $X$ of degree $d\geq 3$ over a $p$-adic number field $K$, an extension $L/K$ and a semistable plane model $\X$ of $X_L$. The implementation is available at \cite{KletusGitHub}; more details on the algorithm will be given in \cite{KletusDiss}.

% !TeX root = reduction_of_plane_quartics.tex

\section{Contracting the \texorpdfstring{$1$}{1}-tails} \label{sec:contraction}

As before, $K$ is a field which is complete with respect to a discrete valuation $v_K$. For simplicity we assume that the residue field $k$ is algebraically closed. We let $p$ denote the characteristic of $k$ (so either $p=0$, or $p$ is a prime number).

Let $X$ be a smooth, projective and absolutely irreducible curve of genus $3$ over $K$. We assume that $X$ is not hyperelliptic. So via the canonical embedding, $X$ can be realized as a smooth plane quartic, $X\subset\PP^2_K$.

By the Semistable Reduction Theorem (\cite{DeligneMumford69}) there exists a finite extension $L/K$ such that $X_L$ has a unique model $\X$ whose special fiber $\X_s$ is a stable curve. This is called the {\em stable model} of $X$. Its special fiber $\Xb:=\X_s$ does not depend on the choice of the extension $L/K$. It is called the {\em stable reduction} of $X$.

We say that $X$ has {\em semistable reduction} if the above holds for $L=K$.
We say that $X$ has {\em potentially good} (resp.\ {\em potentially hyperelliptic reduction}) if $\Xb$ is smooth (resp.\ {\em hyperelliptic}). If $L=K$ we can drop the word {\em potentially}.

\vspace{2ex}
The goal of this section is to prove the following theorem.

\begin{theorem} \label{thm:contraction}
  Assume that $X$ has semistable reduction, and let $\X$ denote the stable model of $X$. Then $X$ has a (unique) GIT-stable plane model $\X_0$ if and only if $\X_s$ is not hyperelliptic. Moreover, if this is the case, then $\X$ dominates $\X_0$, and the map $\X\to\X_0$ contracts the $1$-tails of $\X_s$ to cusps of $\X_{0,s}$, and is an isomorphism everywhere else.
\end{theorem}

\thref{thm:contraction} has the following immediate consequence.

\begin{corollary} \label{cor:contraction}
  Assume that $X$ has a GIT-stable plane model $\X_0$. Then, after replacing $K$ by a suitable finite extension, there exists a modification $\X\to\X_0$ which is the stable model of $X$.
\end{corollary}

In \S\ref{subsec:explicit_cusp_resolution} we explain how Corollary~\ref{cor:contraction} leads to an explicit computation of the stable model in the non-hyperelliptic case, by resolving the cusps of $\X_{0,s}$; the effective construction is carried out in the companion paper~\cite{cusp_resolution}. The rest of this section is devoted to the proof of Theorem~\ref{thm:contraction}. A more detailed version of this proof can be found in \cite{MaxMaster}.

\subsection{Dualizing sheaf and residues}

Let $f: Y \to \Spec A$ be a flat morphism of finite type over a Noetherian base $A$, such that all fibers are Cohen-Macaulay and have pure dimension 1. The theory of duality defines a coherent sheaf $\omega_f := \Hc^{-1}(f^!A)$, which is the unique non-zero cohomology group of the complex $f^!(A)$. We often write $\omega_{Y/A}$ if we regard $Y$ as an $A$-scheme. If $f$ is additionally proper, $\omega_f$ can be characterized as the representing object for the functor $\mathcal{F} \mapsto H^1(Y, \mathcal{F})^\vee$ from the category of coherent sheaves on $Y$ to the category of $A$-modules; this is the statement of Serre duality. The construction of $\omega_f$ is compatible with flat base change and commutes with restriction to any open subscheme.

If $C$ is a proper, reduced curve over $k$, then $\omega_{C/k}$ can be described as the sheaf of \emph{regular differentials} \cite[Theorem 5.2.3]{conrad2000grothendieck}. In the case that $C$ has at most nodes ($A_1$) and cusps ($A_2$) as singularities, its dualizing sheaf is invertible and may be described as follows: let $\nu: \tilde{C} \to C$ be the normalization of $C$, let $x_i, y_i$ be the points of $\tilde{C}$ such that $\nu(x_i) = \nu(y_i)$ are the nodes of $C$, and let $p_j$ be the points of $\tilde{C}$ such that $\nu(p_j)$ are the cusps of $C$. Then $\omega_{C/k}$ is the sheaf of 1-forms $\eta$ on $\tilde{C}$ regular except for simple poles at the $x_i$'s and $y_i$'s and double poles at the $p_j$'s, and with $\Res_{x_i}(\eta) + \Res_{y_i}(\eta) = 0$ and $\Res_{p_j}(\eta) = 0$ (cf.\ \cite[Proposition 3.15]{MaxMaster}).

For our purpose we are interested in ordinary cusps.

\begin{lemma} \label{lem:dualizing_cusp_sheaf}
    Let $C'$ be a reduced proper curve over $k$, and let $S \subset C'$ be a finite set of ordinary cusps. Let $\nu: C \to C'$ be the partial normalization of $C'$ at $S$, and let $P := \nu^{-1}(S)$ be the set of preimages of these cusps. Identifying $P$ with the divisor given by the sum of its points, we have $\nu^*\omega_{C'/k} \simeq \omega_{C/k}(2P)$, and $\omega_{C'/k}$ is the subsheaf of $\nu_*\big(\omega_{C/k}(2P)\big)$ consisting of local sections $\eta$ that satisfy $\Res_{p}(\eta) = 0$ for all $p \in P$.
\end{lemma}

\begin{proof}
    The construction of the dualizing sheaf is local. Away from $S$, the morphism $\nu$ is an isomorphism. At the cusps $x \in S$, both statements follow directly from the local description of regular differentials given above.
\end{proof}

Reversing our perspective and starting from a pointed curve $(C,P)$, we make the following definition.

\begin{definition} \label{def:pinched_curve}
    In the situation of \thref{lem:dualizing_cusp_sheaf}, we call $C'$ the \emph{pinched curve} associated to the pair $(C, P)$. The curve $C'$ together with the morphism $\nu$ is unique up to unique isomorphism. It can be formally constructed from $(C, P)$ by taking the topological space of $C$ and defining the structure sheaf via the local rings $\OO_{C', p} := k + \mathfrak{m}_{C, p}^2$ for $p \in P$, and $\OO_{C', x} := \OO_{C, x}$ elsewhere \cite{serre}.
\end{definition}

Next, we turn to the arithmetic setting over the valuation ring $\OO_K$. Let $X$ be a smooth plane quartic over $K$ (or more generally a smooth, projective, geometrically irreducible curve over $K$), and let $\X$ be a normal model over $\OO_K$ with the special fiber reduced. Denote by $j: U \hookrightarrow \X$ the inclusion of the smooth locus of the $\OO_K$-scheme $\X$; its complement is a finite set of closed points in the special fiber. Then $\omega_{\X/\OO_K} \simeq j_*(\Omega^1_{U/\OO_K})$. From this description, we can deduce for a node $x$ on the special fiber, using the formal neighborhood $\widehat{\OO}_{\X,x} \simeq \OO_K[\![u,v]\!]/(uv-\pi^d)$, that the completed stalk of $\omega_{\X/\OO_K}$ at $x$ is free of rank 1 generated by $du/u = -dv/v$.

The theory of residues and regular differentials can be extended to such models $\X$. We explain only the facts necessary for the proof of \thref{thm:contraction}.

A \emph{flag} on $\X$ is a pair $\xi=(x,y)$, where $y$ is a codimension $1$ point (an irreducible curve $C = \overline{\{y\}}$) and $x$ is a closed point with $x \in C$. The \emph{formal branches} of $C$ at $x$ correspond to the height-1 prime ideals $\mathfrak{P}_1, \dots, \mathfrak{P}_n$ of the completion $\widehat{\OO}_{\X,x}$ lying over the prime ideal defining $y$.

To each branch $\mathfrak{P}_i$, we associate a complete discrete valuation field (CDVF) $L_i$, which is the fraction field of the completion of the localized ring $(\widehat{\OO}_{\X,x})_{\mathfrak{P}_i}$. This defines an algebra $\widehat{K(X)}_\xi := \prod_{i=1}^n L_i$. Each field $L_i$ is equipped with a \emph{local residue map} $\Res_{L_i} \colon \Omega_{L_i/K}^{1,\mathrm{cts}} \to K$, where $\Omega_{L_i/K}^{1,\mathrm{cts}} := \Omega^1_{K(X)/K} \otimes_{K(X)} L_i$. The definition depends on the flag type:
\begin{itemize}
    \item If $\xi$ is \emph{horizontal} ($C \not\subset \X_s$), then $L_i \simeq l_i(\!(t_i)\!)$ where $l_i$ is a finite extension of $K$. For $\omega = \sum a_j t_i^j \, dt_i$, we define $\Res_{L_i}(\omega) = \Tr_{l_i/K}(a_{-1})$.
    \item If $\xi$ is \emph{vertical} ($C \subset \X_s$), then $L_i \simeq K\{\!\{t_i\}\!\}$ (as $\X_s$ is reduced). For $\omega = \sum a_j t_i^j \, dt_i$, we define $\Res_{L_i}(\omega) = -a_{-1}$.
\end{itemize}
The \emph{residue map} for the flag $\xi$ is the sum of these local contributions, obtained by composing the natural map with the sum of local residues:
\[
\Res_{\xi} \colon \Omega^1_{K(X)/K} \to \Omega^1_{K(X)/K} \otimes_{K(X)} \widehat{K(X)}_\xi \simeq \prod_i \Omega_{L_i/K}^{1,\mathrm{cts}} \xrightarrow{\sum \Res_{L_i}} K.
\]
The residue map unifies the concepts of residues on the generic and special fibers. The residue at a closed point $y \in X$ is defined by taking the classical residue in $\kappa(y)$ and applying the field trace $\Tr_{\kappa(y)/K}: \kappa(y) \to K$. A direct consequence of the definition is then:

\begin{proposition}[{{\cite[Prop. 3.27]{MaxMaster}}}] \label{prop:connection_generic_residue}
Let $\xi = (x,y)$ be a flag on $\X$ such that $\overline{\{y\}}$ is horizontal. Then for any $\eta \in \Omega^1_{K(X)/K}$ we have
\[
    \Res_\xi(\eta) = \Res_{y}(\eta) \in K,
\]
where the right-hand side is the residue at the closed point $y \in X$.
\end{proposition}

A differential $\eta \in \Omega^1_{K(X)/K}$ is called \emph{integral at $y$} (a codimension 1 point of $\X$) if it belongs to the $\OO_{\X, y}$-module $\Omega^1_{\OO_{\X,y}/\OO_K} \subset \Omega^1_{K(X)/K}$. This module is the stalk of the dualizing sheaf $\omega_{\X/\OO_K}$ at $y$.
There is a well-defined reduction map $\Omega^1_{\OO_{\X,y}/\OO_K} \to \Omega^1_{\kappa(y)/k}$ to the space of meromorphic differentials on the curve $\overline{\{y\}}$.
The connection between the arithmetic residue on the model and the geometric residue on the special fiber is given by the following proposition.

\begin{proposition}[{{\cite[Prop. 3.29]{MaxMaster}}}] \label{prop:connection_special_residue}
    Let $\xi = (x,y)$ be a vertical flag.
    If $\eta \in \Omega^1_{K(X)/K}$ is integral at $y$, then $\Res_\xi(\eta) \in \OO_K$.
    Moreover, the reduction of the residue modulo $\pi$ is given by
    \[
        \overline{\Res_\xi(\eta)} = -\Res_{x}(\overline{\eta}) \in k,
    \]
    where the right-hand side is the residue on the (possibly singular) curve $C := \overline{\{y\}}$ over $k$, and $\overline{\eta}$ is the image of $\eta$ under the reduction map.
\end{proposition}

This residue map satisfies two key reciprocity laws.

\begin{theorem}[Reciprocity Laws, cf. {{\cite[Theorem 3.33]{MaxMaster}}}] \label{thm:reciprocity_laws}
For any meromorphic differential form $\omega \in \Omega^1_{K(X)/K}$, the following hold:
\begin{enumerate}
    \item \label{thm:reciprocity_laws_i} \emph{Reciprocity Around a Point:} For any closed point $x \in \X$,
    \[
    \sum_{\overline{\{y\}} \ni x} \Res_{(x,y)}(\omega) = 0.
    \]
    \item \label{thm:reciprocity_laws_ii} \emph{Reciprocity Along a Vertical Curve:} For any vertical curve $\overline{\{y\}} \subset \X$, the sum converges in $K$ and
    \[
    \sum_{x \in \overline{\{y\}}} \Res_{(x,y)}(\omega) = 0.
    \]
\end{enumerate}
\end{theorem}

\subsection{Construction of the GIT-Stable Model} \label{subsec:construction_of_GIT_model}

We now prove the forward direction of  \thref{thm:contraction}.
Assume the smooth plane quartic $X$ has semistable reduction over $K$, and let $\X$ be its stable model. Let $\Xb = \X_s$ be its special fiber, which is a stable curve over $k$. We assume that $\Xb$ is \emph{non-hyperelliptic}.
Our goal is to construct a morphism $\phi: \X \to \PP^2_{\OO_K}$ whose image is a GIT-stable plane model.

\label{page:thickness}
The stable model $\X$ is a normal scheme, as its special fiber $\Xb$ is reduced. Its singularities are located precisely at the nodes of $\X_s$. Locally at such a node $x \in \Xb$, the completed local ring is $\widehat{\OO}_{\X,x} \simeq \OO_K[\![u,v]\!]/(uv - \pi^d)$ for a unique integer $d \ge 1$, the \emph{thickness} of the node. The model $\X$ is nonregular at $x$ if and only if $d > 1$.
Since $\X$ is a normal scheme, we may use the theory of Weil divisors. A Weil divisor $D \in Z^1(\X)$ is Cartier if and only if it is locally principal, i.e.\ $D$ gets sent to the trivial class under the canonical map $Z^1(\X) \to \bigoplus_x \Cl(\OO_{\X,x})$. This condition is only relevant at the nonregular points $x$ of $\X$, and depends only on the completion $\widehat{\OO}_{\X,x}$. This is because $\X$ is an excellent scheme (being proper over the complete DVR $\OO_K$), which ensures that the completion is normal and the natural homomorphism $\Cl(\OO_{\X,x}) \hookrightarrow \Cl(\widehat{\OO}_{\X,x})$ is injective.

By \thref{lem:core_tail_lemma}, the stable reduction $\Xb$ decomposes into its core $\Xb_c$ and $r$ 1-tails $\Xb_1, \dots, \Xb_r$, where $r \in \{0, 1, 2, 3\}$. Since $\Xb$ is assumed non-hyperelliptic, Theorem \ref{thm:hyperelliptic_classification_g3} shows that the core is in fact $2$-inseparable. 
Let $x_i$ be the attachment node of $\Xb_i$ and let $d_i$ be its thickness. We define the Weil divisor
\[
D := \sum_{i=1}^r d_i \Xb_i.
\]
This divisor is Cartier. We check this locally at the nodes $x_i$, using the completion $\widehat{\OO}_{\X,x_i} \simeq \OO_K[\![u,v]\!]/(uv - \pi^{d_i})$. If $\Xb_i$ corresponds to the local component $C_i = V(u, \pi)$, then $D$ is locally $d_i C_i$. This is equal to the principal divisor $(u)$, so $D$ is locally principal and therefore Cartier.

The dualizing sheaf $\omega_{\X/\OO_K}$ is invertible since $\Xb$ is a stable curve and hence is a Gorenstein curve. We define the line bundle
\[
\LL := \omega_{\X/\OO_K}(D).
\]
Its restriction to the generic fiber is $\LL|_X \simeq \omega_{X/K} \simeq \OO_{\PP^2_K}(1)|_X$, its global sections form a three dimensional $K$-vector space and since $\Spec K \to \Spec \OO_K$ is flat we have $H^0(\X, \LL) \otimes_{\OO_K} K \cong H^0(X, \LL|_X)$. Thus the $\OO_K$-module of global sections $H^0(\X, \LL)$, being a finitely generated and torsion-free module over the DVR $\OO_K$, is free of rank $3$.
A choice of $\OO_K$-basis $s_0, s_1, s_2$ for $H^0(\X, \LL)$ defines a rational map $\phi: \X \dashrightarrow \PP^2_{\OO_K}$, which restricts to the canonical embedding on the generic fiber $X$.

To verify that $\phi$ is a morphism, we must show that $\LL$ is generated by its global sections. This property may be checked fiberwise. Since $\LL$ is generated on the generic fiber $X$ (where $\phi$ is a closed immersion), it suffices to check this on the special fiber $\Xb$. Let $\LLb := \LL|_{\X_s}$, and let $V \subset H^0(\Xb, \LLb)$ be the $3$-dimensional $k$-vector space spanned by the restrictions $\bar{s}_i$ of the $\OO_K$-basis. In other words, $V$ is the image of the injective canonical map $H^0(\X, \LL) \otimes_{\OO_K} k \to H^0(\Xb, \LLb)$ -- which is in general not surjective. We must therefore show that the linear series $(\LLb, V)$ is base-point-free.

Let $\LLb_c := \LLb|_{\Xb_c}$ and $\LLb_i := \LLb|_{\Xb_i}$ be the restrictions to the core and tails. We analyze $\LL = \omega_{\X/\OO_K}(D)$ locally at a node $x_i$. Retaining the notation from our local description of the dualizing sheaf above, $\omega_{\X/\OO_K}$ is locally generated by $u^{-1}du$, and $D = d_i \Xb_i$ corresponds to the principal divisor $(u)$. The line bundle $\OO_\X(D)$ is thus locally generated by $u^{-1}$, and $\LL$ is locally generated by $\omega_0 = (u^{-1}du) \otimes u^{-1} = du/u^2$.
Restricting this generator $\omega_0$ to the core $\Xb_c$ (where $u$ is a local parameter) gives $du/u^2$, a differential with a double pole, yielding the isomorphism $\LLb_c \simeq \omega_{\Xb_c/k}(\sum 2x_i)$.
On the tail $\Xb_i$ (where $v$ is the local parameter), $\LL$ is locally generated by $\omega_0 = -dv/\pi^{d_i}$. Via the isomorphism $\LL \xrightarrow{\cdot \pi^{d_i}} \LL(-(\pi^{d_i}))$, this generator $\omega_0$ is identified with $-dv$. The reduction of $-dv$ to the special fiber is a non-zero regular differential. This yields $\LLb_i \simeq \omega_{\Xb_i/k}$. As $\Xb_i$ is a 1-tail and hence of arithmetic genus 1, $\omega_{\Xb_i/k} \simeq \OO_{\Xb_i}$.

A global section $s \in H^0(\Xb, \LLb)$ can be identified with a tuple $(s_c, s_1, \dots, s_r) \in H^0(\Xb_c, \LLb_c) \oplus \bigoplus_i H^0(\Xb_i, \LLb_i)$ satisfying the gluing condition $s_c(x_i) = s_i(x_i)$ at each attachment node $x_i$. Since $\LLb_i \simeq \OO_{\Xb_i}$, each $s_i \in H^0(\Xb_i, \LLb_i) \simeq k$ is a constant. This constant is uniquely determined by the value of $s_c$ at $x_i$. Thus, any section $s_c \in H^0(\Xb_c, \LLb_c)$ extends uniquely to a global section $s \in H^0(\Xb, \LLb)$, and the restriction map $\iota: H^0(\Xb, \LLb) \to H^0(\Xb_c, \LLb_c)$ given by $s \mapsto s_c$ is an isomorphism. By the Riemann-Roch theorem, the dimension $h^0(\Xb, \LLb) = h^0(\Xb_c, \LLb_c)$ is:
\[
h^0(\Xb_c, \LLb_c) =
\begin{cases}
    h^0(\Xb_c, \omega_{\Xb_c/k}) = g(\Xb_c) = 3 & \text{if } r=0 \\
    g(\Xb_c) + \deg(\sum 2x_i) - 1 + h^1(\Xb_c, \LLb_c) = r+2 & \text{if } r \ge 1
\end{cases}
\]
Here $h^1(\Xb_c, \LLb_c) = h^0(\Xb_c, \OO_{\Xb_c}(-\sum 2x_i)) = 0$ for $r \ge 1$ as the divisor has negative degree.
Let $W := \iota(V) \subset H^0(\Xb_c, \LLb_c)$ be the $3$-dimensional image of $V$ under the restriction isomorphism $\iota$.

The linear series $(\LLb, V)$ is base-point-free if and only if it is so on the core $\Xb_c$ and on each tail $\Xb_i$. Since $\LLb_i \simeq \OO_{\Xb_i}$, a section $s \in V$ whose core component $s_c = \iota(s)$ does not vanish at $x_i$ restricts to a non-zero constant section on $\Xb_i$, which is nowhere vanishing. Thus, base-point-freeness of $(\LLb_c, W)$ implies base-point-freeness of $(\LLb, V)$.
If this holds, the induced morphism $\phib: \Xb \to \PP^2_k$ contracts each tail $\Xb_i$ to a single point . It therefore suffices to show that the linear series $(\LLb_c, W)$ is base-point-free; this series then induces the morphism $\phib_c := \phib|_{\Xb_c} : \Xb_c \to \PP^2_k$ on the core.

To analyze this linear series, we first provide an explicit characterization of the $3$-dimensional subspace $W$.

\begin{proposition} \label{prop:W_characterization}
The subspace $W$ is characterized as the space of sections with vanishing residues at the attachment points:
\[
W = \{ s_c \in H^0(\Xb_c, \LLb_c) \mid \Res_{x_i}(s_c) = 0 \text{ for } i=1, \dots, r \}.
\]
\end{proposition}

\begin{proof}
Let $W_{\mathrm{Res}}$ denote the subspace on the right-hand side of the equation in the proposition. We first show $\dim(W_{\mathrm{Res}}) = 3$.
The case $r=0$ is trivial, as $W_{\mathrm{Res}} = H^0(\Xb_c, \omega_{\Xb_c/k})$ which has dimension $g(\Xb_c) = 3$.
For $r \ge 1$, the ambient space $H^0(\Xb_c, \LLb_c)$ has dimension $r+2$. By the Residue Theorem on $\Xb_c$, any section $s_c \in H^0(\Xb_c, \LLb_c)$ satisfies $\sum_{i=1}^r \Res_{x_i}(s_c) = 0$. Thus, the $r$ linear conditions defining $W_{\mathrm{Res}}$ are subject to one relation, imposing at most $r-1$ constraints. By an elementary Riemann-Roch calculation, these constraints can be shown to be independent (see \cite[Lemma 3.44]{MaxMaster}), so $\dim(W_{\mathrm{Res}}) = (r+2) - (r-1) = 3$.

Since $\dim(W) = 3 = \dim(W_{\mathrm{Res}})$, it suffices to prove the inclusion $W \subseteq W_{\mathrm{Res}}$. Let $s_c \in W$. By definition, $s_c$ is the component on $\Xb_c$ of the reduction $\bar{\tilde{s}} \in H^0(\Xb, \LLb)$ of some global section $\tilde{s} \in H^0(\X, \LL)$. We must show $\Res_{x_i}(s_c) = 0$. By \thref{prop:connection_special_residue}, this is equivalent to showing the arithmetic residue $\Res_{(x_i, y_c)}(\tilde{s}) = 0$, where $y_c$ is the generic point of $\Xb_c$.

Since $\tilde{s} \in H^0(\X, \LL)$ and $\LL|_X \simeq \omega_{X/K}$, its restriction to $X$ is regular, so by \thref{prop:connection_generic_residue} we have $\Res_\xi(\tilde{s}) = 0$ for all horizontal flags $\xi$. By Reciprocity Around a Point (\thref{thm:reciprocity_laws}~\ref{thm:reciprocity_laws_i}) at $x_i$, the sum of residues over all flags at $x_i$ is zero. As $x_i$ lies only on $\Xb_c$ and the tail $\Xb_i$ (with generic point $y_i$), this simplifies to
\[
\Res_{(x_i, y_c)}(\tilde{s}) + \Res_{(x_i, y_i)}(\tilde{s}) = 0.
\]
By Reciprocity Along a Vertical Curve (\thref{thm:reciprocity_laws}~\ref{thm:reciprocity_laws_ii}) applied to the tail $\Xb_i$, $\sum_{x \in \Xb_i} \Res_{(x, y_i)}(\tilde{s}) = 0$. For any $x \in \Xb_i \setminus \{x_i\}$, the only vertical curve is $\Xb_i$, so Reciprocity Around a Point at $x$ implies $\Res_{(x, y_i)}(\tilde{s}) = 0$. The sum thus reduces to $\Res_{(x_i, y_i)}(\tilde{s}) = 0$.
This forces $\Res_{(x_i, y_c)}(\tilde{s}) = 0$, which proves the inclusion.
\end{proof}

The problem is thus reduced to analyzing the $3$-dimensional linear series $(\LLb_c, W)$ on the core $\Xb_c$. To this end, let $C'$ be the pinched curve (\thref{def:pinched_curve}) associated to the pointed core $(\Xb_c, \{x_1, \dots, x_r\})$, and let $\nu: \Xb_c \to C'$ denote the corresponding partial normalization.

Since $C'$ has at most nodes and cusps as singularities, it is a Gorenstein curve, meaning its dualizing sheaf $\omega_{C'/k}$ is invertible. By \thref{lem:dualizing_cusp_sheaf}, we have the sheaf isomorphism $\nu^*\omega_{C'/k} \simeq \omega_{\Xb_c/k}(\sum 2x_i) = \LLb_c$. Since the pointed core is stable, the line bundle $\LLb_c$ is ample. Because $\nu$ is a finite surjective morphism, this implies that $\omega_{C'/k}$ itself is ample. In the language of \cite[Definition 0.1]{catanese1982pluricanonical}, this means $C'$ is \emph{canonically positive}, which is the baseline assumption throughout \cite[Section 3]{catanese1982pluricanonical}.

Furthermore, taking global sections of the subsheaf inclusion in \thref{lem:dualizing_cusp_sheaf} canonically identifies $H^0(C', \omega_{C'/k})$ with the subspace of $H^0(\Xb_c, \LLb_c)$ consisting of sections with vanishing residues at the attachment points $\{x_1, \dots, x_r\}$. By \thref{prop:W_characterization}, this subspace is precisely $W$. Consequently, the linear series $(\LLb_c, W)$ corresponds exactly to the complete canonical system $|\omega_{C'/k}|$. Since $C'$ is a $2$-inseparable, canonically positive Gorenstein curve, $|\omega_{C'/k}|$ is base-point-free by a theorem of Catanese \cite[Theorem D]{catanese1982pluricanonical}. Therefore, $\phib: \Xb \to \PP^2_k$ is a morphism, which factors as $\phib_c = \Phi_{C'} \circ \nu$, where $\Phi_{C'}$ is the canonical map of $C'$.

Since the $1$-tails are hyperelliptic, \cite[Theorem 4.6]{hyperelliptic} shows that $\Xb$ is hyperelliptic if and only if the pointed core
$(\Xb_c,\{x_1,\ldots,x_r\})$ is hyperelliptic. Hence, by our main assumption,
the pointed core is non-hyperelliptic. By \cite[Theorem 4.13]{hyperelliptic}, this ensures that the canonical map $\Phi_{C'}$ is a closed immersion. Consequently, the special fiber $\X_{0,s}$ of the plane model $\X_0 := \phi(\X)$ (the scheme-theoretic image of $\phi$) is exactly the image $\phib(\Xb) = \Phi_{C'}(C') \simeq C'$. Since $C'$ is a plane quartic whose only singularities are nodes and cusps, $\X_{0,s}$ is a GIT-stable curve. This completes the proof of the forward direction of \thref{thm:contraction}. \qedhere

\subsection{From GIT-Stable to Stable Model}

To prove the converse direction of \thref{thm:contraction}, we have to show the following. Let $X$ be a smooth plane quartic over $K$. Assume that a GIT-stable plane model $\X_0$ of $X$ exists over $K$. Then after replacing $K$ by a suitable finite extension, $X$ admits a stable model over $K$, the stable model $\X$ dominates $\X_0$, and the stable reduction $\X_s$ is non-hyperelliptic. 

By Proposition \ref{prop:semistable}, the special fiber $\X_{0,s}$ of $\X_0$ is reduced and has at most nodes and cusps as singularities. We let $x_1,\ldots,x_r\in\X_{0,s}$ denote the cusps.

\begin{lemma} \label{lem:stable_hull}
  After replacing $K$ by a suitable finite extension, there exists a (unique) minimal semistable model $\X$ of $X$ dominating $\X_0$. Moreover, the following holds.
  \begin{enumerate}
  \item
    The map $\pi:\X\to\X_0$ is an isomorphism except over the cusps.
  \item 
    The strict transform $C_0$ of $\X_{0,s}$ is the partial normalization of $\X_{0,s}$ at the cusps. In particular, $C_0$ is semistable and the map $C_0\to\X_{0,s}$ is a birational homeomorphism.
  \item 
    For each cusp $x_i\in\X_{0,s}$, the exceptional divisor $C_i:=\pi^{-1}(x_i)$ is an irreducible semistable curve of arithmetic genus $1$, intersecting $C_0$ in a unique point $y_i$ (the inverse image of $x_i$), and this point is a node of $C:=\X_s$.
\end{enumerate}
\end{lemma}

\begin{proof}
This is a standard argument.	
By the Semistable Reduction Theorem we may assume that $X$ has semistable reduction over $K$. Now \cite[Theorem 2.3 and Proposition 2.14]{liu2006stable} shows that there exists a (unique) minimal semistable model $\X$ of $X$ dominating $\X_0$ (the {\em stable hull} of $\X_0$). The remaining claims follow easily from the semistability and minimality of $\X$. 
\end{proof}

To finish the proof of \thref{thm:contraction}, it remains to show that $\X_s$ is stable and not hyperelliptic. The special fiber $\X_{0,s} \subset \PP^2_k$ is canonically embedded as a plane quartic. By \cite[Proposition 4.14]{hyperelliptic}, its partial normalization $(C_0, \{y_1, \dots, y_r\})$ is a stable, non-hyperelliptic, $2$-inseparable curve. Since each $(C_i, \{y_i\})$ is also stable (being irreducible of genus $1$ with one marked point), the glued curve $\X_s$ is stable \cite[Lemma 2.2]{hyperelliptic}. Finally, because the $2$-inseparable component $(C_0, \{y_1, \dots, y_r\})$ is not hyperelliptic, the entire curve $\X_s$ cannot be hyperelliptic by \cite[Theorem 4.6]{hyperelliptic}.

\subsection{Explicit resolution of cusps}
\label{subsec:explicit_cusp_resolution}

We return to the problem of explicit semistable reduction.
Let \(X\subset \PP^2_K\) be a smooth plane quartic given by an explicit equation.
After a finite extension of \(K\), one may compute, using
\cite{SternWewers,KletusDiss} and \S\ref{sec:plane_models},
a GIT-semistable plane model
\[
\mathcal X_0\subset \PP^2_{\mathcal O_K}
\]
of \(X\).  In the non-hyperelliptic reduction case this model is GIT-stable.
Then \(\mathcal X_{0,s}\) is reduced and has only nodes and cusps
(Proposition~\ref{prop:semistable}).  Nodes already occur in the stable special
fiber.  Each cusp, on the other hand, has to be replaced by a \(1\)-tail.
By Corollary~\ref{cor:contraction}, after a further finite extension this is
achieved by a modification
\[
\mathcal X\longrightarrow \mathcal X_0
\]
which is an isomorphism away from the cusps of \(\mathcal X_{0,s}\).
The existence proof of Theorem~\ref{thm:contraction} is not constructive; the
explicit local construction is the content of the companion paper
\cite{cusp_resolution}.

We recall the local input and output of that construction.  Let
\(P\in \mathcal X_{0,s}\) be a cusp.  After a linear change of coordinates over
\(\mathcal O_K\), we may assume that \(P=(0:0:1)\), and we work on the affine
chart \(z=1\).  Write the local equation as
\[
f(x,y)=\sum_{i,j} a_{i,j}x^iy^j,\qquad a_{i,j}\in\mathcal O_K .
\]
In admissible coordinates one has
\[
\bar f(x,y)=y^2-x^3+
\text{terms of \((2,3)\)-weighted degree \(>6\)} ,
\]
equivalently, after the usual normalization,
\[
a_{0,2}\in\mathcal O_K^\times,\qquad
a_{3,0}\in\mathcal O_K^\times,\qquad
v_K(a_{i,j})>0\quad\text{for }2i+3j<6 .
\]
For such coordinates set
\begin{equation} \label{eq:t_max}
t_{\max}(x,y)
:=
\min_{2i+3j<6}
\frac{v_K(a_{i,j})}{6-2i-3j}
\in \mathbb Q_{>0}.
\end{equation}
After a finite extension we choose \(\Pi\) with
\(v_K(\Pi)=t_{\max}(x,y)\).  The relevant modification is the weighted blow-up
of the ideal \((\Pi,x,y)\) with weights \((1,2,3)\).

\begin{theorem}[{{\cite[Theorem 1.1]{cusp_resolution}}}]
	\label{thm:explicit_cusp_resolution}
	After replacing \(K\) by a finite separable extension, there exist admissible
	local coordinates \((x,y)\) at \(P\) such that the weighted blow-up of
	\((\Pi,x,y)\) with weights \((1,2,3)\), for
	\(v_K(\Pi)=t_{\max}(x,y)\), is the stable local resolution of the cusp \(P\).
	Its exceptional divisor is an irreducible semistable curve of arithmetic genus
	one, meeting the strict transform of \(\mathcal X_{0,s}\) in one ordinary node.
	The thickness of this node is \(t_{\max}(x,y)\).
	
	More explicitly, the exceptional
	component is the Weierstrass cubic
	\[
	y^2z+\bar{a}_{1,1}xyz+\bar{a}_{0,1}yz^2
	=
	x^3-\bar{a}_{2,0}x^2z
	-\bar{a}_{1,0}xz^2
	-\bar{a}_{0,0}z^3.
	\]
	where
	\[
	  \bar a_{i,j}:=\overline{\Pi^{2i+3j-6}a_{i,j}} \in k.
	\]
	The required field extension, the coordinates \((x,y)\), and equations for the
	weighted blow-up charts are effectively computable from the coefficients of
	\(f\).
\end{theorem}

Applying Theorem~\ref{thm:explicit_cusp_resolution} independently to all cusps
of the GIT-stable special fiber, and then passing to a common finite extension of
\(K\), gives the stable model of \(X\).  This is exactly the local step used in
the implementation: the routine \texttt{resolve\_cusp} computes the local field
extension, the thickness \(t_{\max}\), and the exceptional Weierstrass cubic; the
global routine \texttt{stable\_reduction\_of\_quartic} combines these local
resolutions and records the resulting component graph and reduction type.
See \cite[\S6]{cusp_resolution} for the implementation details.  Examples are
given in the next section.

% !TeX root = reduction_of_plane_quartics.tex

\section{Example computations} \label{sec:examples}

In this final section we illustrate the method with explicit computations.  We
first discuss one example in detail, showing how the implementation produces
the stable reduction data.  We then report on random experiments, which test
the procedure on larger families of smooth plane quartics over \(\mathbb{Q}\).

\subsection{A Ciani quartic}

We return to the example from \cite{SternWewers}, \S 6.2. Set $K:=\QQ_2^\nr$ and consider the smooth quartic $X\subset\PP^2_K$ given by the equation
\[
  F = x_{1}^{4} + 2 x_{0}^{3} x_{2} + x_{0} x_{1}^{2} x_{2} + 2 x_{0} x_{2}^{3}.
\]
Note that $X$ has two commuting involutions,
\[
   \iota_1:(x_0,x_1,x_2)\mapsto (x_0,-x_1,x_2), \quad \iota_2:(x_0,x_1,x_2)\mapsto (x_2,x_1,x_0). 
\]
It is thus a {\em Ciani quartic}, see \cite{Ciani1}. 

Let $L/K$ be a finite extension which contains an element $\theta\in L$ with the property that $v_L(\theta^2+2) > 2$. For instance, we could take the quadratic extensions $K(\sqrt{2})$ or $K(\sqrt{-2})$ (the third quadratic and ramified extension of $K$, $K(\sqrt{-1})$, does not work!). As we have seen in \cite[\S 6.2]{SternWewers}, $X_L$ has a GIT-stable model $\X_0$, given by the equation 
\[
    F_1:= \frac{1}{4}\cdot F(x_0,\theta x_0+\theta^2x_1+\theta x_2, x_2).
\]
Its special fiber $\X_{0,s}$ is the quartic with equation
\[
   \bar{F}_1 = x_0^4 +  x_0x_1^2x_2 + x_0^2x_2^2 + x_2^4.
\]
It is a reduced and irreducible curve with three singularities:
\begin{itemize}
\item
  the node $P_0=(0:1:0)$, with tangent cone $x_0x_2$, and
\item 
  two cusps, $P_1=(\bar{\zeta}:0:1)$ and $P_2=(\bar{\zeta}+1:0:1)$, with tangent cones $(x_0+\bar{\zeta}x_1)^2=0$ and $(x_0+(\bar{\zeta}+1)x_1)^2=0$, respectively. 
\end{itemize}
Corollary \ref{cor:contraction} predicts that, if we choose the right extension $L/K$, there exists a modification
\[
     \pi:\X\to\X_0,
\]
which is a certain weighted blow-up in the two cusps $P_1,P_2$, such that $\X$ is the stable model of $X_L$. The morphism $\pi$ contracts the $1$-tails $C_1,C_2\subset\X_s$ to the cusps $P_1,P_2\in\X_{0,s}$.
Using our implementation, we can compute a suitable extension $L/K$ and the two $1$-tails, as follows. Sage performs the computations over number fields equipped with a unique $2$-adic valuation; we interpret the output as taking place over the strict completion - in line with our standing hypothesis of algebraically closed residue field.

\medskip
\begin{Verbatim}[breaklines=true, breakanywhere=true]
  sage: from semistable_model.curves.stable_reduction_of_quartics import stable_reduction_of_quartic
  sage: R.<x0,x1,x2>=QQ[]
  sage: F = x1^4 + 2*x0^3*x2 + x0*x1^2*x2 + 2*x0*x2^3
  sage: SR = stable_reduction_of_quartic(F, QQ.valuation(2))
\end{Verbatim}

Now {\tt SR} is an object containing information on the stable reduction of the quartic $X$. For instance, we can retrieve the special fiber and the field of definition of the GIT-stable model $\X_0$.

\medskip
\begin{Verbatim}[breaklines=true, breakanywhere=true]
  sage: SR.git_special_fiber
  Projective Plane Curve with defining polynomial x1^4 + (z2 + 1)*x0^2*x1*x2 + z2*x1^3*x2 + z2*x0^2*x2^2 + (z2 + 1)*x1^2*x2^2 + (z2 + 1)*x2^4 over Finite Field in z2 of size 2^2
  sage: SR.git_extension
  Number Field in a1 with defining polynomial x^4 - 2*x^3 + x^2 - 6*x + 9
\end{Verbatim}

So $\X_0$ is defined over the ring of integers of a quartic extension $L_0/\QQ_2$, with ramification index $2$ and inertia degree $2$. The extension of residue fields is made so that the two cusps $P_1$, $P_2$ become rational points of the special fiber $\X_{0,s}$:

\medskip
\begin{Verbatim}[breaklines=true, breakanywhere=true]
  sage: P1, P2 = SR.tail_data.keys()    # the two cusps
  sage: P1
  (z2 : 1 : 1)
  sage: P2
  (z2 + 1 : z2 : 1)
\end{Verbatim}

\noindent
For each cusp $P\in\X_{0,s}$, its {\em tail datum} is a tuple $v_L,t,E,e$, where 
\begin{itemize}
\item
  $v_L$ is the $p$-adic valuation on a field extension $L_1/L_0$ necessary to resolve the cusps $P$ (see \S \ref{subsec:explicit_cusp_resolution}),
\item 
  $t$ is the thickness of the node connecting the tail to the core (equivalently, the parameter $t=t_{\max}(x,y)$ from \eqref{eq:t_max}),
\item 
  $E$ is the tail component, and
\item 
  $e$ is a positive integer, such that the resolution of the cusp $P$ is defined over any extension $L/L_1$ with ramification index $e$. 
\end{itemize}
Here is what happens at the cusp $P_1$:

\medskip
\begin{Verbatim}[breaklines=true, breakanywhere=true]
  sage: v_L, t, E, e = SR.tail_data[P1]  # the tail datum for P1
  sage: v_L.domain()
  Number Field in alpha with defining polynomial a^8 + (8*a1^3 + 8*a1 + 16)*a^7 + (8*a1^3 + 8*a1^2 + 16*a1 - 8)*a^6 + 16*a1^2*a^5 + (8*a1^2 + 32*a1 - 12)*a^4 + (48*a1^3 - 16*a1 + 64)*a^3 + (16*a1^3 + 40*a1 + 60)*a^2 + (-24*a1^3 - 48*a1^2 + 8*a1 + 32)*a + 56*a1^3 - 16*a1 + 60 over its base field
  sage: t
  1/6
  sage: E
  Projective Plane Curve over Finite Field of size 2 defined by x^3 + y^2*z + y*z^2
  sage: e
  3
\end{Verbatim}

This shows that we can resolve the cusp $P_1$ over a totally ramified extension $L/L_0$ of degree $24$. It is constructed as a relative extension $L/L_1/L_0$ such that the admissible coordinate system $(x,y)$ from Theorem \ref{thm:explicit_cusp_resolution} is defined over $L_1$ and such that $L$ has an element $\Pi$ with valuation $t=1/6$. The tail component $C_1$ is the cubic
\begin{equation} \label{eq:C_1}
   C_1:\: x_1^3 + x_0^2x_2 + x_0x_2^2 = 0.
\end{equation}

Since the involution $\iota_2$ swaps the two cusps, it is sufficient to resolve one of them (see also the proof of Proposition \ref{prop:monodromy} below). The resolution of $P_2$ can therefore be defined over the same extension $L/L_0$, and the tail component $C_2$ is isomorphic to $C_1$.

\smallskip
All together, we see that the stable reduction of the quartic $X$ has combinatorial type  {\tt 0nee} (Proposition \ref{prop:classification_g=3}):
\smallskip
\begin{Verbatim}[breaklines=true, breakanywhere=true]
  sage: SR.reduction_type
  '0nee' 
\end{Verbatim}
See Figure \ref{fig:contraction_special_fiber} for a schematic representation of the contraction map. 

\begin{figure}
    \centering
    
    \newsavebox{\topFiberBox}
    \newsavebox{\bottomFiberBox}
    
    \sbox{\topFiberBox}{\resizebox{4cm}{!}{\HobbyCurve{0nee}}}
    \sbox{\bottomFiberBox}{\resizebox{4cm}{!}{\HobbyCurve{0ncc}}}

    \begin{tikzpicture}[node distance=1.8cm] 
        
        % --- Top Panel: Source Fiber ---
        \node (topNode) [inner sep=0pt] {\usebox{\topFiberBox}};
        \node [right=0.3cm of topNode] {$\mathcal{X}_s$};

        % Labels für C1 und C2
        \node [font=\small] at ([xshift=-1.0cm, yshift=1.1cm]topNode.center) {$C_1$};
        \node [font=\small] at ([xshift=1.0cm, yshift=1.1cm]topNode.center) {$C_2$};

        % --- Bottom Panel: Target Fiber ---
        \node (bottomNode) [below=of topNode, inner sep=0pt] {\usebox{\bottomFiberBox}};
        \node [right=0.3cm of bottomNode] {$\mathcal{X}_{0,s}$};

        % Labels für P1 und P2
        \node [font=\small] at ([xshift=-1.1cm, yshift=0.25cm]bottomNode.center) {$P_1$};
        \node [font=\small] at ([xshift=1.1cm, yshift=0.25cm]bottomNode.center) {$P_2$};

        % --- Contraction Morphism Arrow ---
        \draw[->, thick] ([yshift=-0.3cm]topNode.south) -- node[midway, right=0.1cm] {$\pi_s$} ([yshift=0.5cm]bottomNode.north);

    \end{tikzpicture}
    \caption{The special fiber of the contraction morphism} \label{fig:contraction_special_fiber}
\end{figure}
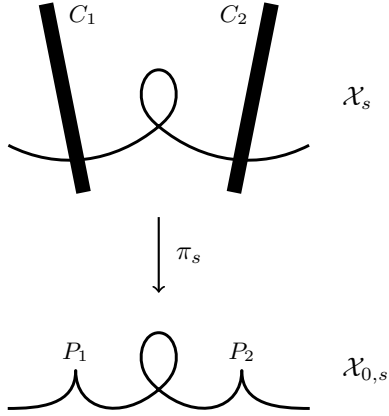

\medskip

The following proposition shows that $L/L_0$ is the minimal extension of $L_0$ over which $X$ has semistable reduction. (Recall that we consider $L,L_0,\ldots$ as finite extensions of $K=\QQ_2^\nr$.)

\begin{proposition} \label{prop:monodromy}
  The extension $L/L_0$ is a totally ramified Galois extension, with Galois group $\Gal(L/L_0)\cong\SL_2(3)$ of order $24$. It acts faithfully on $\X_s$, fixing all three components $C_0,C_1,C_2$. It acts trivially on $C_0$. For $i=1,2$, the action on $C_i$ induces an isomorphism
  \[
       \Gal(L/L_0)\cong\Aut_k(C_i,x_i).
  \]	
\end{proposition}

\begin{proof}
The local stable resolution of the cusp $P_1$ which produced the extension $L/L_0$ is {\em strongly rigidified} in the sense of
\cite[Definition~5.6]{cusp_resolution}: indeed, the exceptional cubic
\eqref{eq:C_1} is in strong rigidified Weierstrass normal form (\cite[Definition 2.5]{cusp_resolution}).  Hence
\cite[Theorem~5.9, together with Proposition~5.8]{cusp_resolution}, applied
to the cusp \(P_1\) over the base field \(L_0\), shows that \(L/L_0\) is the
minimal extension over which the local stable resolution of \(P_1\) is defined.
Equivalently, the induced monodromy action on the tail \(C_1\) is faithful.
By the involution \(\iota_2\), the same statement holds for the cusp \(P_2\).
	
Let $\tilde{L}/L$ denote the Galois closure of $L/L_0$, and let  $G\subset\Aut_k(\X_s)$ denote the image of $\Gal(\tilde{L}/L_0)$ under the monodromy action. The action of $G$ on $\X_s$ commutes with the contraction map $\X_s\to\X_{0,s}$. Since the GIT-stable model $\X_0$ was originally defined over $L_0$, $G$ acts trivially on $\X_{0,s}$ and hence trivially on $C_0$. It also follows that $G$ fixes the two tails $C_1,C_2$. We obtain an embedding
\begin{equation} \label{eq:group_embedding}
    G\hookrightarrow H_1\times H_2,
\end{equation}
where $H_i:=\Aut_k(C_i,\tilde{P}_i)$. 

The involution $\iota_2$ on $X$ induces an involution $\iota_{2,s}$ of $\X_s$ which commutes with the action of $G$. It is easy to check that $\iota_{2,s}$ swaps the two tails $C_1,C_2$. It follows that \eqref{eq:group_embedding} becomes a diagonal embedding, once we identify $H_1$ and $H_2$ via the isomorphism induced by $\iota_{2,s}$, and $G$ with a subgroup of $H_1$ via the restriction of its action to $C_1$. 

It is well known that the automorphism group of the cubic \eqref{eq:C_1} is a group of order $24$, isomorphic to $\SL_2(3)$ (it is the supersingular elliptic curve with $j=0$ in Characteristic $2$, see e.g.\ \cite[Theorem III.10.1 and Exercise 3.12]{SilvermanAEC}). In particular this order is equal to the degree of $L/L_0$. We conclude that $G=H_1$ and that $L/L_0$ is already a Galois extension, i.e. $\tilde{L}=L$. Now everything is proved.
\end{proof}

Even though the extension $L/L_0$ is minimal with the property that $X$ has semistable reduction over $L$, this is not true for the extension $L/K$. Indeed, $L/K$ has degree $48$ which is larger than the order of the monodromy group. It follows that there exists a subextension $L'\subset L$ of index $2$ over which $X$ has semistable reduction. This is a drawback of our general approach: in the first step we need some extension over which we have a GIT-stable model. But there is no natural candidate, so the extension that we find will often be disjoint from the minimal extension over which the curve has semistable reduction.

\subsection{Random experiments}
\label{subsec:random-experiments}

The preceding subsection treated a single curve in detail.  We now report on
larger-scale computations for random smooth plane quartics over \(\mathbb{Q}\),
with respect to the \(p\)-adic valuations for \(p=2,3,5\).  The aim is to demonstrate that
the implementation can handle many examples and to record the stable reduction
types observed in these experiments.

For each run, the output directory contains the
per-sample records, the aggregate statistics and the run parameters.  Thus the
tables below can be checked directly from the archived data files. For instructions on how to reproduce the data, see \href{https://github.com/swewers/quartic_reduction_data/tree/main#}{the README file of the repository}.

\begin{table}[ht]
	\centering
	\small
	\begin{tabular}{c|c|c|c}
		Run & \(p\) & output directory & code commit \\ \hline
		A & \(2\) &
		\href{https://github.com/swewers/quartic_reduction_data/tree/55b34e4c7fe958575829aa9eeb3b893f34483e5c/p2/p2_500_22-5-2026}{\texttt{p2\_500\_22-5-2026}} &
		\href{https://github.com/swewers/StabilityFunction/tree/2841d56b5d8f71e076de5999d79db37f8e7272d0}{\texttt{2841d56b5d8f}} \\
		B & \(3\) &
		\href{https://github.com/swewers/quartic_reduction_data/tree/55b34e4c7fe958575829aa9eeb3b893f34483e5c/p3/p3_500_22-5-2026}{\texttt{p3\_500\_22-5-2026}} &
		\href{https://github.com/swewers/StabilityFunction/tree/2841d56b5d8f71e076de5999d79db37f8e7272d0}{\texttt{2841d56b5d8f}} \\
		C & \(5\) &
		\href{https://github.com/swewers/quartic_reduction_data/tree/55b34e4c7fe958575829aa9eeb3b893f34483e5c/p5/p5_500_22-5-2026}{\texttt{p5\_500\_22-5-2026}} &
		\href{https://github.com/swewers/StabilityFunction/tree/2841d56b5d8f71e076de5999d79db37f8e7272d0}{\texttt{2841d56b5d8f}}
	\end{tabular}
	\caption{Archived random-experiment runs.  The output directories contain the corresponding
		\texttt{data.jsonl}, \texttt{stats.json}, and \texttt{run\_parameters.json}
		files.}
	\label{tab:random-runs}
\end{table}

In all three runs we used the same sampling parameters: \(500\) smooth
quartics, \(8\) randomly chosen monomials, coefficient bound \(10\), seed
\(55496\), and no \(p\)-adic coefficient bias.  Only the prime \(p\) was varied.

The status counts are summarized in Table~\ref{tab:random-status}.  Here
\texttt{ok} means that the computation produced a stable reduction type.  The
status \texttt{hyperelliptic} means that the computation reached the strictly
semistable GIT case; by Theorem~\ref{thm:contraction}, this is precisely the
case where the present method does not proceed to a stable model.  The status
\texttt{fail} records an exception raised during the computation, while
\texttt{timeout} records a computation which exceeded the prescribed time
limit, here \(120\,\mathrm{s}\) per quartic.

\begin{table}[ht]
	\centering
	\small
	\begin{tabular}{c|c|c|c|c|c|c}
		Run & \(p\) & smooth quartics & \texttt{ok} & \texttt{hyperelliptic}
		& \texttt{fail} & \texttt{timeout} \\ \hline
		A & \(2\) & \(500\) & \(295\) & \(189\) & \(0\) & \(16\) \\
		B & \(3\) & \(500\) & \(422\) & \(60\) & \(1\) & \(17\) \\
		C & \(5\) & \(500\) & \(465\) & \(35\) & \(0\) & \(0\)
	\end{tabular}
	\caption{Status counts for the random experiments.}
	\label{tab:random-status}
\end{table}

\FloatBarrier

The only \texttt{fail} entry was caused by a \emph{PARI stack overflow}.
Together with the \(33\) timeouts, this appears to be due to coefficient swell in
intermediate computations.  These cases reflect limitations of the current
implementation rather than additional geometric reduction types, and could
likely be handled by further optimization of the code.

Among the successful computations, we recorded the stable reduction type using
the notation of Proposition~\ref{prop:classification_g=3}.  The distribution is
shown in Table~\ref{tab:random-reduction-types}.  We only include samples with
status \texttt{ok} in this table.

\begin{table}[h]
	\centering
	\small
	\begin{tabular}{c|r|r|r|r}
		reduction type & Run A & Run B & Run C & total \\ \hline
		\texttt{0----0} & 7 & 3 & 0 & 10 \\
		\texttt{0---0e} & 38 & 26 & 2 & 66 \\
		\texttt{0---0m} & 2 & 0 & 0 & 2 \\
		\texttt{0---0n} & 0 & 6 & 6 & 12 \\
		\texttt{0eee} & 0 & 19 & 0 & 19 \\
		\texttt{0mee} & 0 & 1 & 0 & 1 \\
		\texttt{0nee} & 2 & 0 & 0 & 2 \\
		\texttt{0nne} & 0 & 2 & 0 & 2 \\
		\texttt{0nnn} & 0 & 2 & 3 & 5 \\
		\texttt{1---0} & 42 & 37 & 31 & 110 \\
		\texttt{1ee} & 25 & 5 & 0 & 30 \\
		\texttt{1me} & 1 & 0 & 0 & 1 \\
		\texttt{1ne} & 3 & 7 & 8 & 18 \\
		\texttt{1nn} & 8 & 11 & 16 & 35 \\
		\texttt{2e} & 26 & 49 & 37 & 112 \\
		\texttt{2m} & 1 & 2 & 1 & 4 \\
		\texttt{2n} & 61 & 94 & 92 & 247 \\
		\texttt{3} & 77 & 146 & 264 & 487 \\
		\texttt{BRAID} & 0 & 0 & 1 & 1 \\
		\texttt{CAVE} & 2 & 12 & 4 & 18 \\ \hline
		total & 295 & 422 & 465 & 1182
	\end{tabular}
	\caption{Stable reduction types among the successful random examples.}
	\label{tab:random-reduction-types}
\end{table}

\FloatBarrier

These computations give a practical stress test for the two-step procedure
described in \S\ref{subsec:explicit_cusp_resolution}: first compute a
GIT-semistable plane model, and then, in the GIT-stable case, resolve the cusps
to obtain the stable model. The observed frequencies depend on the sampling
parameters and should therefore not be interpreted as intrinsic densities in
moduli.  They nevertheless provide a reproducible collection of examples,
including the raw equations and the computed reduction data.

All computations reported in this subsection were carried out with SageMath
10.8, using Python 3.12.3, and the version of \texttt{StabilityFunction}
specified in Table~\ref{tab:random-runs}. The computations were run on a laptop with an Intel Core i7-10510U processor
and 16\,GiB RAM, under Ubuntu~24.04.4~LTS.

\vspace{2ex}\noindent
{\bf Data and code availability statement:}
The data generated and analyzed in Section~\ref{subsec:random-experiments}
are available at
\url{https://github.com/swewers/quartic_reduction_data}.  The repository contains
the raw output files, aggregate statistics, and run parameters for the random
experiments.  The computations were carried out using the SageMath package
\texttt{StabilityFunction}, available at
\url{https://github.com/swewers/StabilityFunction}; the precise code commit is
recorded in Table~\ref{tab:random-runs}.

\vspace{2ex}\noindent
{\bf Conflict of interest statement:}
The authors declare that they have no conflict of interest.

\bibliographystyle{plain}
\bibliography{references}

\medskip
{\small
\noindent
\textit{Address:}
Institute of Algebra and Number Theory,
Ulm University,
Helmholtz\-strasse 18,
89081 Ulm, Germany

\smallskip

\noindent
\textit{Email addresses:}
\texttt{max.schwegele@uni-ulm.de},
\texttt{kletus.stern@uni-ulm.de},\\
\texttt{stefan.wewers@uni-ulm.de}

}

\end{document}